\documentclass[pdftex,11pt]{amsart} 

\usepackage{array,amsfonts,amssymb,amsmath,mathtools,latexsym,upref,enumerate,cool,multicol}
\usepackage[margin=2.0cm]{geometry}
\usepackage{verbatim} 
\usepackage[english]{babel}
\usepackage[latin1]{inputenc}

\usepackage{mathtools}

\usepackage{color,tikz,pgf}
\usetikzlibrary{shapes, intersections, calc, patterns, positioning,
  matrix, arrows, fit, backgrounds, 
  decorations.pathmorphing, decorations.markings}

\newtheorem{lemma}[equation]{Lemma}
\newtheorem{prop}[equation]{Proposition}
\newtheorem{thm}[equation]{Theorem}
\newtheorem{cor}[equation]{Corollary}

\newtheorem{defn}[equation]{Definition}

\theoremstyle{definition}

\newtheorem{exmp}[equation]{Example}

\newtheorem{rmk}[equation]{Remark}

\newtheorem{notation}[equation]{Notation}

\numberwithin{equation}{section}

\newcommand{\Z}{\mathbf{Z}}

\newcommand{\Q}{\mathbf{Q}}
\newcommand{\C}{\mathbf{C}}

\newcommand{\ve}[1]{\varepsilon_{#1}}

\newcommand{\func}[3]{\mbox{$#1 \colon #2 \to #3$}}
\newcommand{\qt}[2]{#1 \backslash #2}
\newcommand{\dqt}[3]{#1 \backslash #2 / #3}

\newcommand{\Ob}[1]{\mathrm{Ob}(#1)}

\newcommand{\m}{morphism}

\newcommand{\we}{weighting}

\newcommand{\Euc}{Euler characteristic}

\newcommand{\act}[2]{#1 \backslash #2 // #2}

\newcommand{\mynote}[1]{\noindent{\textcolor{red}{\textbf{[#1]}}}}

\title{Groupoid $G$-spans and matrices over group rings}

\author{Joachim Kock}
\address{Departament de Matem\`atiques, Universitat Aut\`onoma de 
Barce\-lona, E--08193 Bellaterra, Spain}
\email{Joachim.Kock@uab.cat}

\author{Jesper M. M\o{}ller}
\address{Matematisk Institut, Universitetsparken 5, DK--2100 K\o{}benhavn, 
Denmark}
\email{moller@math.ku.dk}


\usepackage[bookmarks=true,bookmarksopen=false]{hyperref}
\hypersetup{
   pdftitle = {},
   pdfauthor = {Jesper Michael MÃÂÃÂÃÂÃÂ¸ller},
   pdfpagemode = {UseOutlines},
   pdfstartview = {FitH},
   pdfborder = {0 0 0},
   backref = {true},
   colorlinks = {true},
   urlcolor = {blue},
   citecolor = {blue},
   linkcolor = {blue},
   pdftoolbar = {true},
}

\begin{document}

\date{\today}
\begin{abstract}
When $G$ is a finite abelian group, we define $G$-spans of groupoids
and their associated matrices with entries in the group ring $\Q G$ and show 
that composition of spans corresponds to multiplication of matrices.
\end{abstract}
\maketitle
\tableofcontents




\section{Introduction}
\label{sec:introduction}

For any span of finite {\em sets}, i.e.\ any diagram of finite sets of the form
\begin{center}
    \begin{tikzpicture}[double distance = 1.5pt]
       \matrix (dia) [matrix of math nodes, column sep=25pt, row
       sep=25pt]{
         M & T \\ S \\};
       \draw[->] (dia-1-1) -- node[above]{$R$}(dia-1-2);
        \draw[->] (dia-1-1) -- node[left]{$L$}(dia-2-1);
   \end{tikzpicture}
\end{center}
let  $[M] \colon S \times T \to \Z_{\geq 0}$ be the function
given by
\begin{equation*}
  [M](c,d) = | \dqt cMd |, \qquad c \in S,\, d \in T
\end{equation*}
where $| \dqt cMd|$ is the cardinality of the two-sided fibre 
$\dqt cMd = \{x \in M \mid c=Lx,\, Rx=d\}$. 
For example,
 \begin{equation*}
   [M] =
   \begin{pmatrix}
     3 & 0 & 2 \\ 2 & 4 & 4
   \end{pmatrix}
\end{equation*}
if $S=\{c_1,c_2,c_3\}$, $T=\{d_1,d_2\}$ and
 the set
$M$ looks like this
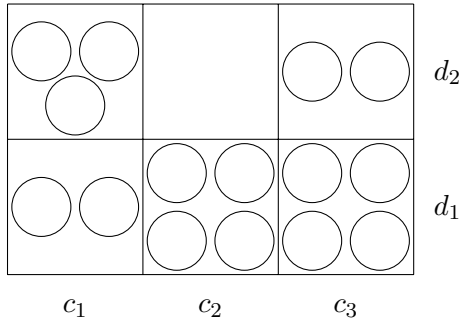
\begin{center}
        \begin{tikzpicture}
   [every circle node/.style={draw}, scale=0.9]
      \draw[step=2.0]  (0,0) grid (6,4);

      \node at (1,-0.5) {$c_1$};  \node at (3,-0.5) {$c_2$};
      \node  at (5,-0.5) {$c_3$};

     \node at (6.5,1.0) {$d_1$};  \node at (6.5,3.0) {$d_2$};

     \draw (0.5,1) node[circle] {$\phantom{g_1}$};
     \draw (1.5,1) node[circle] {$\phantom{g_1}$};

     \draw (0.5,3.3) node[circle] {$\phantom{g_1}$};
     \draw (1.5,3.3) node[circle] {$\phantom{g_2}$};
     \draw (1.0,2.5) node[circle] {$\phantom{g_2}$};
    
     \draw (2.5,1.5) node[circle] {$\phantom{g_1}$};
     \draw (3.5,1.5) node[circle] {$\phantom{g_1}$};
     \draw (2.5,0.5) node[circle] {$\phantom{g_2}$};
     \draw (3.5,0.5) node[circle] {$\phantom{g_1}$};

     \draw (4.5,3) node[circle] {$\phantom{g_1}$};
     \draw (5.5,3) node[circle] {$\phantom{g_2}$};

     \draw (4.5,1.5) node[circle] {$\phantom{g_3}$};
     \draw (5.5,1.5) node[circle] {$\phantom{g_1}$};
     \draw (4.5,0.5) node[circle] {$\phantom{g_2}$};
     \draw (5.5,0.5) node[circle] {$\phantom{g_2}$};
\end{tikzpicture}
  \end{center}
  when its elements, indicated by circles, are grouped into two-sided
  fibres.  
An appealing, and elementary, feature of this construction is that
span composition,
defined by pull-backs,
corresponds to matrix multiplication. However, 
only matrices with non-negative integral entries can be realized by set
spans. So the $(1 \times 1)$-matrices $(\frac{1}{2})$
and $(-1)$ are unrealizable by spans of sets.

Baez, Hoffnung, and Walker \cite[Theorem~34]{BHW2010}
went further and associated to
any span of finite {\em groupoids}
\begin{center}
    \begin{tikzpicture}[double distance = 1.5pt]
       \matrix (dia) [matrix of math nodes, column sep=25pt, row
       sep=25pt]{
         M & T \\ S \\};
       \draw[->] (dia-1-1) -- node[above]{$R$}(dia-1-2);
        \draw[->] (dia-1-1) -- node[left]{$L$}(dia-2-1);
      \end{tikzpicture} 
\end{center} 
the function $[M] \colon \pi_0(S) \times \pi_0(T) \to \Q_{\geq
  0}$ given by
\begin{equation*}
  [M](c,d) = |T(d,d))|^{-1} \chi(\dqt cMd), \qquad c \in \pi_0(S),\, d \in \pi_0(T)
\end{equation*}
where $T(d,d)$ is the auto\m\ group of $d$ in $T$ and $\chi(\dqt cMd)$
the \Euc, or groupoid cardinality, of the two-sided homotopy fibre of $M$ over $c$ and $d$. 
For example, if $\pi_0(S)=\{c_1,c_2,c_3\}$, $\pi_0(T)=\{d_1,d_2\}$, and
the two-sided homotopy fibres of $M$
and their connected subgroupoids, indicated by circles, look like 
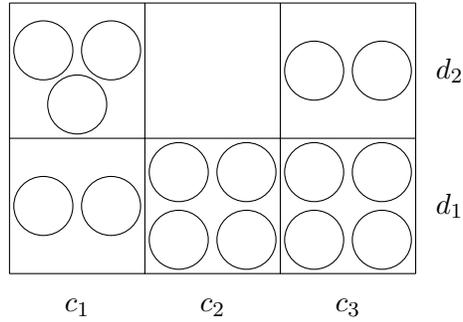
\begin{center}
        \begin{tikzpicture}
   [every circle node/.style={draw}, scale=0.9]
      \draw[step=2.0]  (0,0) grid (6,4);

      \node at (1,-0.5) {$c_1$};  \node at (3,-0.5) {$c_2$};
      \node  at (5,-0.5) {$c_3$};

     \node at (6.5,1.0) {$d_1$};  \node at (6.5,3.0) {$d_2$};

     \draw (0.5,1) node[circle] {$\phantom{g_1}$};
     \draw (1.5,1) node[circle] {$\phantom{g_1}$};

     \draw (0.5,3.3) node[circle] {$\phantom{g_1}$};
     \draw (1.5,3.3) node[circle] {$\phantom{g_2}$};
     \draw (1.0,2.5) node[circle] {$\phantom{g_2}$};
    
     \draw (2.5,1.5) node[circle] {$\phantom{g_1}$};
     \draw (3.5,1.5) node[circle] {$\phantom{g_1}$};
     \draw (2.5,0.5) node[circle] {$\phantom{g_2}$};
     \draw (3.5,0.5) node[circle] {$\phantom{g_1}$};

     \draw (4.5,3) node[circle] {$\phantom{g_1}$};
     \draw (5.5,3) node[circle] {$\phantom{g_2}$};

     \draw (4.5,1.5) node[circle] {$\phantom{g_3}$};
     \draw (5.5,1.5) node[circle] {$\phantom{g_1}$};
     \draw (4.5,0.5) node[circle] {$\phantom{g_2}$};
     \draw (5.5,0.5) node[circle] {$\phantom{g_2}$};
\end{tikzpicture}
  \end{center}
  then $[M](c_2,d_1)$ is $|T(d_1,d_1)|^{-1}$ times a sum of \Euc s of
  four subgroupoids while $[M](c_2,d_2)=0$ as the two-sided homotopy
  fibre $\dqt{c_2}M{d_2}$ is empty. Again, composition of groupoid
  spans corresponds to matrix multiplication
  \cite[Theorem~34]{BHW2010} but now the entries of the matrices
  are non-negative rational numbers. The 
  $(1 \times 1)$-matrix $(\frac{1}{2})$ is now realizable  but 
  $(-1)$ is still not realizable.  

  The purpose of this note is to present a more general construction
  allowing for the realization of a wider class of matrices.  In our
  paper \cite{kock_moller2026} we introduced a setting where one can realize matrices
  with negative entries. In the present paper we go further and get to
  complex numbers. To get there, it is natural to first develop the
  theory with values in group rings.

  In this note, $G$ is a finite abelian group, $BG$ the groupoid
  consisting of a single object with auto\m\ group $G$, and $\Q_{\geq
    0}G$ the rational group semi-ring.
  The input
  data for our construction is a $G$-span of groupoids by which we
  understand a diagram of groupoids and groupoid \m s
\begin{center}
\begin{tikzpicture}
         \matrix (dia) [matrix of math nodes, column sep=80pt, row sep=40pt]{
      M & T\\
      S & BG \\};
    \draw[->] (dia-1-1) --  node[above]{$R$} (dia-1-2);
     \draw[->] (dia-1-1) --  node[left]{$L$} (dia-2-1);
      \draw[->] (dia-1-2) --  node(V)[right]{$V$} (dia-2-2);
      \draw[->] (dia-2-1) --  node[below]{$H$} (dia-2-2);
      \path (dia-2-1) -- node[midway]{$\varepsilon \colon HL \implies VR$}(dia-1-2); 
  \end{tikzpicture}
  \end{center}
with a natural transformation $\ve{} \colon \Ob{M} \to G$ from the functor $HL$ to the functor $VR$. 
The two-sided homotopy fibres of $M$ may now look like
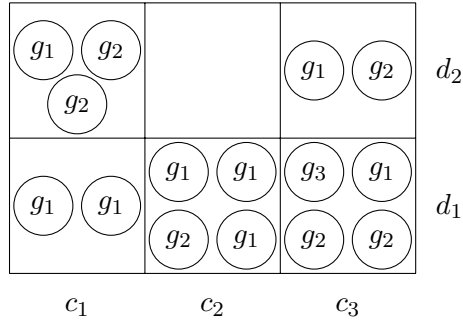
\begin{center}
        \begin{tikzpicture}
   [every circle node/.style={draw}, scale=0.9]
      \draw[step=2.0]  (0,0) grid (6,4); 

      \node at (1,-0.5) {$c_1$};  \node at (3,-0.5) {$c_2$};
      \node  at (5,-0.5) {$c_3$};

     \node at (6.5,1.0) {$d_1$};  \node at (6.5,3.0) {$d_2$};

     \draw (0.5,1) node[circle] {$g_1$};
     \draw (1.5,1) node[circle] {$g_1$};

     \draw (0.5,3.3) node[circle] {$g_1$};
     \draw (1.5,3.3) node[circle] {$g_2$};
     \draw (1.0,2.5) node[circle] {$g_2$};
    
     \draw (2.5,1.5) node[circle] {$g_1$};
     \draw (3.5,1.5) node[circle] {$g_1$};
     \draw (2.5,0.5) node[circle] {$g_2$};
     \draw (3.5,0.5) node[circle] {$g_1$};

     \draw (4.5,3) node[circle] {$g_1$};
     \draw (5.5,3) node[circle] {$g_2$};

     \draw (4.5,1.5) node[circle] {$g_3$};
     \draw (5.5,1.5) node[circle] {$g_1$};
     \draw (4.5,0.5) node[circle] {$g_2$};
     \draw (5.5,0.5) node[circle] {$g_2$};
\end{tikzpicture}
  \end{center}
where the circles,  representing connected subgroupoids,  now carry
labels from the set $G$.
To any $G$-span as above, we associate the function
$[M,\ve{}] \colon \pi_0(S) \times \pi_0(T) \to \Q_{\geq 0}G$ given by
   \begin{equation*}
     [M,\ve{}](c,d) = |T(d,d)|^{-1} \sum_{g \in G} \chi( (\dqt cMd)\{ \dqt c{\ve{}}d = g \})g, \qquad c \in \pi_0(S), \, d \in \pi_0(T)
   \end{equation*}
   The non-negative rational coefficient of $g$ is $|T(d,d)|^{-1}$
   times the \Euc\ of the $g$-labeled subgroupoid
   $(\dqt cMd)\{ \dqt c{\ve{}}d = g \}$ of the two-sided homotopy fibre $\dqt cMd$. For example, if the
   two-sided homotopy fibres of $M$ look as
   above, only $g_1$ and $g_2$ have non-negative coefficients in
   $[M,\ve{}](c_2,d_1) \in \Q_{\geq 0} G$.

The composition of the function $[M,\ve{}] \colon \pi_0(S) \times \pi_0(T) \to
\Q G$ with a complex character $\rho
\colon \Q G \to \C$
is a function
\begin{equation*}
[[M,\ve{}]]  \colon \pi_0(S) \times \pi_0(T) \xrightarrow{[M,\ve{}]} \Q G
\xrightarrow{\rho} \C
\end{equation*}
with values in the complex numbers.  Both of these associated
matrices, $[M,\ve{}]$ and $[[M,\ve{}]]$, translate composition of
$G$-spans to matrix multiplication as stated in
Theorem~\ref{thm:main}.  The values of $[[M,\ve{}]]$ may be complex
numbers, not necessarily non-negative rational numbers.
For example, if $\sqrt[4]{1}=\{ \pm 1, \pm i\}$ is the cyclic group of
the complex $4$th roots of unity, then the $\sqrt[4]{1}$-span
   \begin{center}
    \begin{tikzpicture}[double distance = 1.5pt]
       \matrix (dia) [matrix of math nodes, column sep=35pt, row
       sep=25pt]{
         \{1\} & \{1\} \\ \{1\} & B\sqrt[4]{1} \\};
       \draw[->] (dia-1-1) -- (dia-1-2);
       \draw[->] (dia-1-1) --  (dia-2-1);
       \draw[->] (dia-2-1) --  (dia-2-2);
       \draw[->] (dia-1-2) --  (dia-2-2);
      \path (dia-2-1) -- node[midway]{$\ve{}(1) =z$} (dia-1-2);
   \end{tikzpicture}
\end{center}
realizes the $(1 \times 1)$-matrix $[[\{1\}, \ve{}]] = (z)$ for any $z
= \pm 1, \pm i$. (See Example~\ref{exmp:-1matrix} for more details.)
With our construction it is thus possible to realize the
$(1 \times 1)$-matrices $(-1)$ and $(i)$.

We now briefly summarize the contents of this paper. The following
Section~\ref{sec:homotopy-pull-backs} introduces some notation and
recapitulates the notion of homotopy pull-backs of
groupoids. Definition~\ref{defn:funddefn} in
Section~\ref{sec:spans-with-signs} contains the central
definitions of a $G$-span and its associated matrix over the group
semi-ring $\Q_{\geq 0}G$. The main result of this paper,
Theorem~\ref{thm:main}, is that composition of $G$-spans corresponds to
multiplication of the associated matrices. The formal structure of the
category of $G$-spans is outlined in
Section~\ref{sec:functorial-aspects}, and
Proposition~\ref{prop:restrictM} unveils a general property of
matrices associated to $G$-spans. We work out some concrete examples of
matrices generated by $G$-spans of groupoids in the final
Section~\ref{sec:examples}.

\section{Homotopy pull-backs of groupoids}
\label{sec:homotopy-pull-backs}

We recapitulate the theory of homotopy pull-backs of groupoids and
introduce some notation that will be used throughout this note.

\begin{defn}[The homotopy pull-back groupoid {\cite[\S 2]{BHW2010}}] \label{defn:htpypullback}
  Let $M_1 \xrightarrow{R} T \xleftarrow{L} M_2$ be a cospan of
  groupoids \cite[Definition~A.15]{BHW2010}. The homotopy pull-back groupoid, $M_1 \times_T M_2$, is the
  groupoid whose objects are all triples $(a_1,t,a_2)$ where $a_1 \in \Ob{M_1}$,  $a_2 \in \Ob{M_2}$, and $t \in
  T(Ra_1,La_2)$ is a \m\ in $T$ from $Ra_1$ to $La_2$. The \m s 
$(a_1,t,a_2) \to (b_1,u,b_2)$ in $M_1 \times_T M_2$
are all pairs $(m_1,m_2) \in
M_1(a_1,b_1) \times M_2(a_2,b_2)$ such that 
\begin{center}
    \begin{tikzpicture}[double distance = 1.5pt]
       \matrix (dia) [matrix of math nodes, column sep=45pt, row
       sep=35pt]{
         Ra_1 & La_2 \\ Rb_1 & Lb_2 \\};
       \draw[->] (dia-1-1) -- node[above]{$t$}(dia-1-2);
       \draw[->] (dia-1-1) -- node[left]{$R(m_1)$}(dia-2-1);
       \draw[->] (dia-1-2) -- node[right]{$L(m_2)$}(dia-2-2);
       \draw[->] (dia-2-1) -- node[below]{$u$}(dia-2-2);
   \end{tikzpicture}
\end{center}
commutes in $T$.
\end{defn}


The homotopy pull-back square
\begin{center}
      \begin{tikzpicture}[double distance = 1.5pt]
       \matrix (dia) [matrix of math nodes, column sep=45pt, row
       sep=40pt]{
         M_1 \times_T M_2 & M_2 \\ M_1 & T \\};
       \draw[->] (dia-1-1) -- node[above]{$p_2$}(dia-1-2);
       \draw[->] (dia-1-1) -- node[left]{$p_1$}(dia-2-1);
       \draw[->] (dia-1-2) -- node[right]{$L$}(dia-2-2);
       \draw[->] (dia-2-1) -- node[below]{$R$}(dia-2-2);
       \path (dia-2-1) --node[midway]{$Rp_1 \implies  Lp_2$} (dia-1-2);
   \end{tikzpicture}
\end{center}
commutes up to the natural transformation $Rp_1(a_1,t,a_2)=Ra_1 \xrightarrow{t} La_2=Lp_2 (a_1,t,a_2)$.

Alternatively, the homotopy pull-back,
\begin{equation}
  \label{eq:htpypullback01}
  M_1 \times_T M_2 = \int_{(a_1,a_2) \in \Ob{M_1^{\mathrm{op}}
    \times M_2}} T(Ra_1,La_2)
\end{equation}
is the category of elements on the  functor on
$M_1^{\mathrm{op}} \times M_2$ taking the object $(a_1,a_2)$ to the set $T(Ra_1,La_2)$.


\begin{rmk}[Two-sided homotopy pull-backs]\label{rmk:pbofpb}
 From  groupoid \m s $P \xrightarrow{R_1} S \xleftarrow{L} M
 \xrightarrow{R} T \xleftarrow{L_2} Q$ we can form the homotopy
 pull-backs 
\begin{center}
   \begin{tikzpicture}[double distance = 1.5pt]
       \matrix (dia) [matrix of math nodes, column sep=25pt, row sep=25pt]{
       (P \times_S M) \times_M (M \times_T Q) & M \times_T Q & Q  &&   (P \times_S M) \times_T Q & Q  \\
       P \times_S M & M & T &&  P \times_S M & T \\
       P & S \\};
     \draw[->] (dia-1-2) -- (dia-1-3);
     \draw[->] (dia-1-1) -- (dia-1-2);
     \draw[->] (dia-1-1) -- (dia-2-1);
     \draw[->] (dia-2-1) -- (dia-2-2);
     \draw[->] (dia-2-2) -- node[below]{$R$} (dia-2-3);
     \draw[->] (dia-3-1) -- node[above]{$R_1$}(dia-3-2);
     \draw[->] (dia-1-2) -- (dia-2-2);
     \draw[->] (dia-1-3) -- node[right]{$L_2$}(dia-2-3);
     \draw[->] (dia-2-1) -- (dia-3-1);
     \draw[->] (dia-2-2) -- node[right]{$L$}(dia-3-2);

    \draw[->] (dia-1-5) -- (dia-1-6);
    \draw[->] (dia-1-5) -- (dia-2-5);
    \draw[->] (dia-1-6) -- (dia-2-6);
    \draw[->] (dia-2-5) -- (dia-2-6);
    \draw[->] (dia-1-5) -- (dia-1-6);
\end{tikzpicture}
\end{center}
and there are groupoid equivalences
\begin{equation} \label{eq:USMTV}
 \begin{tikzpicture}[>=stealth', baseline=(current bounding box.-2)]
  \matrix (dia) [matrix of math nodes, column sep=25pt, row sep=20pt]{
    (P \times_S M) \times_M (M \times_T Q) &  (P \times_S M) \times_T Q  \\};
  \draw[->] ($(dia-1-1.east)+(0,0.1)$) -- ($(dia-1-2.west)+(0,0.1)$) 
  node[pos=.5,above] {${}$}; 
  \draw[->] ($(dia-1-2.west)-(0,0.1)$) -- ($(dia-1-1.east)-(0,0.1)$)
   node[pos=.5,below] {${}$} ; 
 \end{tikzpicture}
\end{equation}
since the homotopy pull-back of a homotopy pull-back is 
equivalent to
 a homotopy
pull-back along a composite \m . The objects of the two-sided homotopy pull-back
$P \times_S M \times_T
Q$ are $(x,s,a,t,y)$ where $(x,a,y) \in \Ob{P} \times \Ob{M} \times \Ob{Q}$ and $s \in
S(R_1x,La)$, $t \in T(Ra,L_2v)$. The \m s $(x_1,s_1,a_1,t_1,y_1) \to
(x_2,s_2,a_2,t_2,y_2)$ are triples $(u,m,v) \in P(x_1,x_2) \times
M(a_1,a_2) \times Q(y_1,y_2)$ such that the relevant diagrams in $S$
and $T$ commute. There is a groupoid span $P \leftarrow P \times_S M
\times_T Q \to Q$.
\end{rmk}

 \begin{notation}\label{defn:fullinvgroupoid}
Let $T$ be a groupoid, $A \subseteq \Ob{T}$ a set of objects of $T$,
$d \in \Ob{T}$ an object of
$T$, and $\ve{} \colon \Ob{T} \to Y$ a function on the set of objects
of $T$ with values in some set $Y$. 
\begin{enumerate}[(1)]
\item
$T\{A\}$  is the full subgroupoid of $T$ on the objects in $A$ and
$T\{\ve{} = y\} = T\{\ve{}^{-1}y\}$ for any $y \in Y$.
 \item $1\{d\}$ is the subgroupoid of $T$ with $d$ as its only object
   and the identity as its only \m .
\end{enumerate}
\end{notation}

If $c$ is an object of $S$, $d$ an object of $T$, and $S \xleftarrow{L} M
\xrightarrow{R} T$ a groupoid span, the left, right, and two-sided homotopy fibres
\begin{align*}
  &\qt{c}{M} = 1\{c\} \times_S M  = \int_{a \in \Ob{M}} S(c,La) \qquad
  M/d = M \times_T 1\{d\} = \int_{a \in \Ob{M}} T(Ra,d) \\
  &\dqt{c}{M}{d} = 1\{c\} \times_S M \times_T 1\{d\}   =\qt cM
    \times_M M/d  =  \int_{a \in \Ob{M}}
   S(c,La) \times T(Ra,d)
\end{align*}
sit in a diagram of homotopy pull-back squares
\begin{center}
   \begin{tikzpicture}[double distance = 1.5pt]
       \matrix (dia) [matrix of math nodes, column sep=25pt, row
       sep=25pt]{
      \dqt{c}{M}{d} & M/d & 1\{d\} \\
      \qt{c}{M}  & M & T \\
      1\{c\} & S \\};
    \draw[->] (dia-1-1) -- (dia-1-2);
    \draw[->] (dia-1-2) -- (dia-1-3);
    \draw[->] (dia-1-2) -- (dia-2-2);
    \draw[->] (dia-1-1) -- (dia-2-1);
    \draw[->] (dia-2-1) -- (dia-2-2);
    \draw[->] (dia-2-1) -- (dia-3-1);
    \draw[->] (dia-1-3) -- (dia-2-3);
    \draw[->] (dia-2-2) -- node[below]{$R$}(dia-2-3);
    \draw[->] (dia-2-2) --node[right]{$L$} (dia-3-2);
    \draw[->] (dia-3-1) -- (dia-3-2);
  \end{tikzpicture} 
\end{center} 

The \Euc, or groupoid
cardinality, of a finite groupoid $S$
is $\chi(S) = \sum_{c \in \pi_0(S)} |S(c,c)|^{-1}$
\cite[Definition~2.2, Example~2.7]{leinster08}, \cite[Definition~2.4]{BHW2010}.

\begin{lemma} \label{lemma:chiM1TM2}
\cite[Lemma~3.8]{imma_kock_tonks2018} 
  The homotopy pull-back groupoid of Definition~\ref{defn:htpypullback}
  has \Euc\
  \begin{equation*}
   \chi( M_1 \times_T M_2) = \sum_{d \in \pi_0(T)} \chi(M_1/d )  \chi(T\{d\}) \chi(\qt d{M_2})
  \end{equation*}
\end{lemma}
\begin{proof}
Let $k^1_\bullet$ be a co\we\ on $M_1$ and $k^\bullet_2$ a \we\ on $M_2$. 
The \Euc\ of the Grothendieck construction $M_1 \times_T M_2 = \int_{M_1^{\mathrm{op}} \times M_2} T(Ra_1,La_2)$ is
\cite[Proposition~2.8]{leinster08} 
\begin{align*}
  \chi(M_1 \times_T M_2) &= \sum_{a_1,a_2} k^1_{a_1} k^{a_2}_2 |T(Ra_1,La_2)| =
  \sum_{d \in \pi_0(T)} \sum_{\substack{a_1 \in \Ob{M_1},a_2 \in \Ob{M_2} \\ T(Ra_1,d) \neq \emptyset, T(d,La_2) \neq \emptyset}} 
  k^1_{a_1} k^{a_2}_2 |T(Ra_1,La_2)| \\ &=
 \sum_{d \in \pi_0(T)} |T(d,d)|^{-1}  \sum_{\substack{a_1 \in \Ob{M_1},a_2 \in \Ob{M_2} \\ T(Ra_1,d) \neq \emptyset, T(d,La_2) \neq \emptyset}} 
  k^1_{a_1} k^{a_2}_2 |T(Ra_1,d)||T(d,La_2)| \\ &=
  \sum_{d \in \pi_0(T)} \chi(T\{d\}) \sum_{a_1 \in \Ob{M_1}, a_2 \in \Ob{M_2}} 
    k^1_{a_1} k^{a_2}_2 |T(Ra_1,d)||T(d,La_2)|\\ &=
\sum_{d \in \pi_0(T)} \chi(T\{d\}) \big(\sum_{a_1} k_{a_1} |T(Ra_1,d)| \big) \big(\sum_{a_2} k^{a_2} |T(d,La_2)| \big)
  \\ &=\sum_{d \in \pi_0(T)}  \chi(M_1/d) \chi(T\{d\}) \chi(\qt{d}{M_2}) 
\end{align*}
as claimed. For the second equality we used that $|T(Ra_1,La_2)|$ is nonzero only if there is some object $d$ of $T$ such that $d$, $Ra_1$, $La_2$ lie in the same component of $T$. 
\end{proof}

\begin{rmk}[The homotopy fibre versus the full inverse image {\cite[Definition~2.8]{BHW2010}}]
  With reference to $L \colon M \to S$, the full inverse image of an
  object $c$ in $S$ is the groupoid $L^{-1}(c) = M\{ a \in \Ob{M} \mid
  S(c,La) \neq \emptyset\}$. The \Euc\
  of the left homotopy fibre $\qt cM = \qt c{L^{-1}(c)} = \int_{a \in L^{-1}(c)} S(c,La)$ is
  $\chi(\qt cM) = |S(c,c)| \chi(L^{-1}(c))$
  \cite[Proposition~2.8]{leinster08}. We can rewrite this as
$\chi(L^{-1}(c)) =
  \chi(S\{c\}) \chi(\qt cM)$ (and it follows that $\chi(M) =
  \sum_{\pi_0(S)} \chi(S\{c\})\chi(\qt cM)$). The full inverse image
  will not be used in this paper, where we prefer homotopy fibres
  instead, but this remark explains how to switch between the two concepts.
\end{rmk}

\section{$G$-spans of groupoids and their matrices}
\label{sec:spans-with-signs}

In this section we introduce $G$-spans of
groupoids where $G$ is a finite abelian group. To every $G$-span we associate a matrix with coefficients
in the group semi-ring $\Q_{\geq 0}G$
and show that composition
of $G$-spans translates to matrix multiplication
(Definition~\ref{defn:funddefn}, Theorem~\ref{thm:main}). 

$BG$ is the groupoid with $\bullet$ as its only object with auto\m\ group
$BG(\bullet,\bullet)=G$.

\begin{defn} \label{defn:funddefn}
Let $S$, $T$ and $M$ be finite groupoids.
\begin{enumerate}[(1)]

  \item  \label{funddefn:4}
 A $G$-span from $S$ to $T$ with apex $M$ is a diagram of groupoids and functors
 \begin{center} 
  \begin{tikzpicture}
         \matrix (dia) [matrix of math nodes, column sep=80pt, row sep=40pt]{
      M & T\\
      S & BG \\};
    \draw[->] (dia-1-1) --  node[above]{$R$} (dia-1-2);
     \draw[->] (dia-1-1) --  node[left]{$L$} (dia-2-1);
      \draw[->] (dia-1-2) --  node(V)[right]{$V$} (dia-2-2);
      \draw[->] (dia-2-1) --  node[below]{$H$} (dia-2-2);
      \path (dia-2-1) -- node[midway]{$HL \stackrel{\ve{}}{\implies} VR$}(dia-1-2); 
  \end{tikzpicture}
  \end{center}
  and a function $\varepsilon \colon \Ob{M} \to G$ such that
  $\varepsilon(a_2)   (HL)(m) = (VR)(m) \varepsilon(a_1)$ in $G$
for all
  $a_1,a_2 \in \Ob{M}$ and $m \in M(a_1,a_2)$:
  \begin{center}
  \begin{tikzpicture}[double distance = 1.5pt]
       \matrix (dia) [matrix of math nodes, column sep=50pt, row
       sep=40pt]{HLa_1 & VRa_1 \\
HLa_2 & VRa_2 \\};
\draw[->] (dia-1-1) --node[left]{$(HL)(m)$} (dia-2-1);
\draw[->] (dia-1-2) --node[right]{$(VR)(m)$} (dia-2-2);
\draw[->] (dia-1-1) --node[above]{$\ve{}(a_1)$} (dia-1-2);
\draw[->] (dia-2-1) --node[below]{$\ve{}(a_2)$} (dia-2-2);
\end{tikzpicture}
\end{center}

\item \label{funddefn:1b} In the situation of \eqref{funddefn:4}, let
  $c \in \Ob{S}$ and $d \in \Ob{T}$. The homotopy fibre over $c$ is
  the $G$-span with apex  $\qt cM = 1\{c\} \times_S M$  
\begin{center}
  \begin{tikzpicture}[double distance = 1.5pt]
       \matrix (dia) [matrix of math nodes, column sep=75pt, row
       sep=40pt]{
      \qt cM & T \\
      1\{c\} & BG \\};
   \draw[->] (dia-1-1) -- node[above]{$\dot R = Rp_2$} (dia-1-2);
   \draw[->] (dia-1-1) -- (dia-2-1);
     \draw[->] (dia-2-1) -- (dia-2-2);
   \draw[->] (dia-1-2) -- node[right]{$V$}(dia-2-2);
   \path (dia-2-1) -- node[midway]{$ e \stackrel{\qt{c}{\ve{}}}{\implies}
     V\dot R$} (dia-1-2);
\end{tikzpicture}
\end{center}
from $1\{c\}$ to $T$, 
and the homotopy fibre over $d$ is the $G$-span with apex $M/d = M
\times_T 1\{d\}$ 
\begin{center}
  \begin{tikzpicture}[double distance = 1.5pt]
       \matrix (dia) [matrix of math nodes, column sep=75pt, row
       sep=40pt]{
      M/d & 1\{d\} \\
      S & BG \\};
   \draw[->] (dia-1-1) --  (dia-1-2);
   \draw[->] (dia-1-1) -- node[left]{$\dot L = Lp_1$} (dia-2-1);
     \draw[->] (dia-2-1) -- node[below]{$H$} (dia-2-2);
   \draw[->] (dia-1-2) -- (dia-2-2);
   \path (dia-2-1) -- node[pos=0.6]{$H \dot L \stackrel{\ve{}/d}{\implies} e$} (dia-1-2);
\end{tikzpicture}
\end{center}
from $S$ to $1\{d\}$ in the diagram 
 \begin{center}
    \begin{tikzpicture}
         \matrix (dia) [matrix of math nodes, column sep=60pt, row
         sep=40pt]{
     {} & M \times_T 1\{d\} & 1\{d\} \\
      1\{c\} \times_S M & M & T \\ 1\{c\} & S & BG \\};
\draw[->] (dia-2-1) -- node[above]{$p_2$}(dia-2-2);
\draw[->] (dia-2-2) -- node[above]{$R$}(dia-2-3);
\draw[->] (dia-3-1) -- (dia-3-2);
\draw[->] (dia-3-2) -- node[below]{$H$}(dia-3-3);
\draw[->] (dia-2-1) -- node[left]{$p_1$}(dia-3-1);
\draw[->] (dia-2-2) -- node[left]{$L$}(dia-3-2);
\draw[->] (dia-2-3) -- node[right]{$V$}(dia-3-3);
\draw[->](dia-1-2) -- node[above]{$p_2$}(dia-1-3);
\draw[->](dia-1-2) -- node[left]{$p_1$}(dia-2-2);
\draw[->](dia-1-3) -- (dia-2-3);
\path (dia-3-2) -- node[midway]{$HL \stackrel{\ve{}}{\implies}
  VR$}(dia-2-3);
\path (dia-2-2) -- node[midway]{$Rp_1 \stackrel{t}{\implies} p_2$}
(dia-1-3);
\path (dia-3-1) -- node[midway]{$p_1 \stackrel{s}{\implies} Lp_2$} (dia-2-2);
     \end{tikzpicture}
  \end{center}
The objects of $\qt cM$ are pairs $(s,a)$ where $a \in \Ob{M}$ and $s \in S(c,La)$.
The function
$
  \qt c\varepsilon  \colon \Ob{\qt cM}  \to G
$
given by $(\qt c\varepsilon)(s,a)  = \varepsilon(a) H(s)$ makes the diagram
\begin{center}
  \begin{tikzpicture}
           \matrix (dia) [matrix of math nodes, column sep=40pt, row
           sep=30pt]{
Hc & HL(a_1) & VR(a_1) \\
Hc & HL(a_2) & VR(a_2) \\};
\draw[->](dia-1-3) --node[right]{$VR(m)$} (dia-2-3);
\draw[->](dia-1-2) -- node[right]{$HL(m)$}(dia-2-2);
\draw[double](dia-1-1) -- node[left]{$e$}(dia-2-1);
\draw[->](dia-1-2) -- node[above]{$\varepsilon(a_1)$}(dia-1-3);
\draw[->](dia-1-1) -- node[above]{$H(s_1)$}(dia-1-2);
\draw[->](dia-2-2) -- node[below]{$\varepsilon(a_2)$}(dia-2-3);
\draw[->](dia-2-1) -- node[below]{$H(s_2)$}(dia-2-2);
\draw[->] (dia-1-1)  |-
($(dia-1-3.north)+(0,0.6)$) node[pos=0.75,above]{$(\qt
  c{\ve{}})(s_1,a_1)$}-- (dia-1-3);
\draw[->] (dia-2-1) |-  node[pos=0.75,below]{$(\qt c{\ve{}})(s_2,a_2)$}  ($(dia-2-3.south)+(0,-0.6)$) --
(dia-2-3); 
\end{tikzpicture}
\end{center}
commute for any \m\ $m \colon (s_1,a_1) \to (s_2,a_2)$ in $\qt cM$.

The objects of $M/d$ are pairs $(a,t)$ where $a \in \Ob{M}$ and $t \in T(Ra,d)$. The function
$
  \varepsilon/d  \colon \Ob{M/d}  \to G
$
given by $(\varepsilon/d)(a,t) = V(t)\varepsilon(a)$ makes the diagram
\begin{center}   
  \begin{tikzpicture}
           \matrix (dia) [matrix of math nodes, column sep=40pt, row sep=30pt]{
HL(a_1) & VR(a_1) & Vd \\
HL)a_2) & VR(a_2) & Vd \\};
\draw[double](dia-1-3) -- node[right]{$e$}(dia-2-3);
\draw[->](dia-1-2) -- node[above]{$V(t_1)$}(dia-1-3);
\draw[->](dia-2-2) -- node[below]{$V(t_2)$}(dia-2-3);
\draw[->](dia-1-2) --node[left]{$VR(m)$} (dia-2-2);
\draw[->](dia-1-1) -- node[left]{$HL(m)$}(dia-2-1);
\draw[->](dia-1-1) -- node[above]{$\varepsilon(a_1)$}(dia-1-2);
\draw[->](dia-2-1) -- node[below]{$\varepsilon(a_2)$}(dia-2-2);
\draw[->](dia-1-1) |- node[pos=0.75,above]{$(\ve{}/d)(a_1,t_1)$} ($(dia-1-3.north)+(0,0.6)$) -- (dia-1-3);
\draw[->](dia-2-1) |- node[pos=0.75,below]{$(\ve{}/d)(a_2,t_2)$}
 ($(dia-2-3.south)+(0,-0.6)$) -- (dia-2-3);
\end{tikzpicture}
\end{center}
commute for any \m\ $m \colon (a_1,t_1) \to (a_2,t_2)$ in $M/d$.

\item \label{funddefn:2} In the situation of \eqref{funddefn:4}, let
  $c$ be an object of $S$ and $d$ an object of $T$. The two-sided homotopy
  fibre over $(c,d)$ is the 
$G$-span with apex $\qt{c}{M/d} = 1\{c\} \times _S M \times_T 1\{d\}$ from $1\{c\}$ to $1\{d\}$ 
  \begin{center}
    \begin{tikzpicture}
           \matrix (dia) [matrix of math nodes, column sep=65pt, row sep=30pt]{
       \dqt cMd  & 1\{d\} \\ 1\{c\} & BG \\};
      \draw[->] (dia-1-1) -- (dia-1-2);
      \draw[->] (dia-2-1) -- (dia-2-2);
       \draw[->] (dia-1-1) -- (dia-2-1);
       \draw[->] (dia-1-2) -- (dia-2-2);
      \path (dia-2-1) --node[midway]{$ Hc
        \stackrel{\qt c{\ve{}}/d}{\implies} Vd$} (dia-1-2);
    \end{tikzpicture}
  \end{center}
formed by the outer edge of the square 
\begin{center}
  \begin{tikzpicture}
           \matrix (dia) [matrix of math nodes, column sep=45pt, row sep=30pt]{
         \qt{c}{M}/d  & M/d  & 1\{d\} \\
        \qt cM & M & T \\
        1\{c\} & S & BG \\};
   \draw[->] (dia-1-1) -- (dia-1-2);
   \draw[->] (dia-1-2) -- (dia-1-3);
   \draw[->] (dia-2-1) -- (dia-2-2);
   \draw[->] (dia-2-2) --node[above]{$R$} (dia-2-3);
   \draw[->] (dia-3-1) -- (dia-3-2);
   \draw[->] (dia-3-2) --node[below]{$H$} (dia-3-3);

   \draw[->](dia-1-1) -- (dia-2-1);
   \draw[->](dia-1-2) -- (dia-2-2);
   \draw[->](dia-1-3) -- (dia-2-3);

   \draw[->](dia-2-1) -- (dia-3-1);
   \draw[->](dia-2-2) -- node[left]{$L$}(dia-3-2);
   \draw[->](dia-2-3) --node[right]{$V$} (dia-3-3);

   \path (dia-3-2) -- node[midway]{$HL \stackrel{\ve{}}{\implies} VR $}(dia-2-3);
   \end{tikzpicture}
\end{center}
The objects of $\qt cM/d$ are all triples $(s,a,t)$ where $a$ is an
object of $M$ and $s \in S(c,La)$, $t \in T(Ra,d)$.  The function
$ \dqt c{\ve{}}d \colon \mathrm{Ob}( \qt cM/d) \to G $ is
$(\qt c{\varepsilon}/d)(s,a,t)=V(t) \ve{}(a) H(s)$. For any \m\
$m \colon (s_1,a_1,t_1) \to (s_2,a_2,t_2)$ in
$\dqt cMd$, the diagram
\begin{center}
  \begin{tikzpicture}
           \matrix (dia) [matrix of math nodes, column sep=40pt, row sep=30pt]{
Hc & HL(a_1) & VR(a_1) & Vd \\
Hc & HL(a_2) & VR(a_2) & Vd \\};
\draw[double](dia-1-1) -- node[left]{$e$}(dia-2-1);
\draw[->] (dia-1-1) --node[above]{$H(s_1)$} (dia-1-2);
\draw[->] (dia-2-1) --node[below]{$H(s_2)$} (dia-2-2);
\draw[double](dia-1-4) -- node[right]{$e$}(dia-2-4);
\draw[->](dia-1-3) -- node[above]{$V(t_1)$}(dia-1-4);
\draw[->](dia-2-3) -- node[below]{$V(t_2)$}(dia-2-4);
\draw[->](dia-1-3) --node[left]{$VR(m)$} (dia-2-3);
\draw[->](dia-1-2) -- node[left]{$HL(m)$}(dia-2-2);
\draw[->](dia-1-2) -- node[above]{$\varepsilon(a_1)$}(dia-1-3);
\draw[->](dia-2-2) -- node[below]{$\varepsilon(a_2)$}(dia-2-3);
\draw[->](dia-1-1) |- node[pos=0.75,above]{$(\dqt{c}{\ve{}}{d})(s_1,a_1,t_1)$} ($(dia-1-4.north)+(0,0.6)$) -- (dia-1-4);
\draw[->](dia-2-1) |- node[pos=0.75,below]{$(\dqt{c}{\ve{}}{d})(s_2,a_2,t_2)$}
 ($(dia-2-4.south)+(0,-0.6)$) -- (dia-2-4);
\end{tikzpicture}
\end{center}
commutes in $BG$.

\item \label{funddefn:6} 
The  matrix $[M,\ve{}]  \colon \pi_0(S) \times \pi_0(T) \to \Q G$
of  the $G$-span in  Definition~\ref{defn:funddefn}.(\ref{funddefn:4}) is
 \begin{equation*}
 [M,\varepsilon](c,d) =  
 \sum_{g \in G} \chi((\qt cM/d)\{\qt c\varepsilon / d = g\}) \chi(T\{d\}) g
\qquad  c \in \pi_0(S), d \in \pi_0(T)
\end{equation*}

\item  \label{funddefn:4a}
 Composition of a $G$-span from $S_1$ to $T$ with a $G$-span from $T$ to $S_2$ 
  \begin{center}
  \begin{tikzpicture}
      \matrix (dia) [matrix of math nodes, column sep=80pt, row sep=40pt]{
     M_1 & T && M_2 & S_2 \\
     S_1 & BG && T & BG \\};
\draw[->](dia-1-1) -- node[above]{$R_1$} (dia-1-2);
\draw[->](dia-1-1) -- node[left]{$L_1$} (dia-2-1);
\draw[->](dia-1-2) -- node[right]{$V_1$} (dia-2-2);
\draw[->](dia-2-1) -- node[below]{$H_1$} (dia-2-2);
\draw[->](dia-1-4) -- node[above]{$R_2$} (dia-1-5);
\draw[->](dia-1-4) -- node[left]{$L_2$} (dia-2-4);
\draw[->](dia-1-5) -- node[right]{$V_2$} (dia-2-5);
\draw[->](dia-2-4) -- node[below]{$H_2$} (dia-2-5);
\path (dia-2-1) -- node[midway]{$   H_1L_1
  \stackrel{\ve 1}{\implies} V_1R_1$}(dia-1-2);
\path (dia-2-4) -- node[midway]{$ H_2L_2
  \stackrel{\ve 2}{\implies} V_2R_2 $}(dia-1-5);
\end{tikzpicture}
\end{center}
is defined, provided that $V_1=H_2 \colon T \to BG$, as the $G$-span from $S_1$ to $S_2$
 \begin{center}
   \begin{tikzpicture}
       \matrix (dia) [matrix of math nodes, column sep=110pt, row sep=50pt]{
       M_1 \times_T M_2 & S_2 \\
       S_1 & BG \\};
       \draw[->] (dia-1-1) -- node[above]{$\dot R_2 = R_2p_2$}(dia-1-2);
       \draw[->] (dia-2-1) -- node[below]{$H_1$}(dia-2-2);
       \draw[->] (dia-1-1) -- node[left, xshift=-2pt]{$\dot L_1 = L_1p_1$}(dia-2-1);
       \draw[->] (dia-1-2) -- node[right]{$V_2$}(dia-2-2);
       \path (dia-2-1) -- node[midway]{$H_1 \dot L_1 \stackrel{\ve 1 \times_T \ve 2}{\implies} V_2 \dot R_2$} (dia-1-2);
     \end{tikzpicture}  
 \end{center}
 at the outer edge of the homotopy commutative diagram
\begin{center}
  \begin{tikzpicture}[double distance = 1.5pt]
       \matrix (dia) [matrix of math nodes, column sep=50pt, row sep=40pt]{
 M_1 \times_T M_2 & M_2 & S_2 \\
M_1 & T & BG \\
S_1 & BG & BG\\};
\draw[->] (dia-1-1) -- node[above]{$p_2$}(dia-1-2);
\draw[->] (dia-1-2) -- node[above]{$R_2$}(dia-1-3);
\draw[->] (dia-1-1) -- node[left]{$p_1$}(dia-2-1);
\draw[->] (dia-1-2) -- node[right]{$L_2$}(dia-2-2);
\draw[->] (dia-1-3) -- node(V2)[right]{$V_2$}(dia-2-3);
\draw[->] (dia-2-1) -- node[above]{$R_1$}(dia-2-2);
\draw[->] (dia-2-2) -- node[above]{$H_2$}(dia-2-3);
\draw[->] (dia-2-1) -- node[left]{$L_1$}(dia-3-1);
\draw[->] (dia-2-2) -- node[right]{$V_1$}(dia-3-2);
\draw[->] (dia-3-1) -- node[below]{$H_1$}(dia-3-2);
\path (dia-3-1) -- node[midway]{$\varepsilon_1$} (dia-2-2);
\path (dia-2-2) -- node[midway]{$\varepsilon_2$} (dia-1-3);
\path (dia-2-1) -- node[midway]{$t$} (dia-1-2);
\draw [double](dia-3-2)--(dia-3-3);
\draw[double](dia-2-3) -- (dia-3-3);
  \end{tikzpicture}
\end{center}
 where 
$\varepsilon_1 \times_T \varepsilon_2 \colon \Ob{M_1 \times_T M_2} \to G$ is the map $(\varepsilon_1 \times_T \varepsilon_2)(a_1,t,a_2) =  \varepsilon_2(a_2)  V_1(t) \varepsilon_1(a_1)$. 

For any \m\  $(m_1,m_2) \colon (a_1,t,a_2,) \to (b_1,u,b_2)$ in $M_1 \times_T M_2$,
$uR_1(m_1)=L_2(m_2)t$, as expressed by
the commutative diagram in $T$
\begin{center}
  \begin{tikzpicture}
     \matrix (dia) [matrix of math nodes, column sep=50pt, row
     sep=40pt]{
     R_1(a_1) & L_2(a_2) \\
     R_1(b_1) & L_2(b_2) \\};
 \draw[->](dia-1-1) -- node[above]{$t$} (dia-1-2);
 \draw[->](dia-1-1) -- node[left]{$R_1(m_1)$} (dia-2-1);
 \draw[->](dia-2-1) -- node[below]{$u$} (dia-2-2);
 \draw[->](dia-1-2) -- node[right]{$L_2(m_2)$} (dia-2-2);
  \end{tikzpicture}
\end{center}
so that  
\begin{equation} \label{eq:VLalpha}
(V_1L_2)(m_2) V_1(t) = V_1(u) (V_1R_1)(m_1) \quad
(H_2L_2)(m_2) H_2(t) = H_2(u) (H_2R_1)(m_1)
\end{equation}
in $G$. 
The identity
\begin{equation*}
   (\ve{1} \times_T \ve{2})(b_1,u,b_2) (V_2\dot R_2)(m_1,m_2) = (H_1
   \dot L_1)(m_1,m_2)  (\ve{1} \times_T \ve{2}) (a_1,t,a_2)
\end{equation*}
holds because
\begin{center}
  \begin{tikzpicture}[double distance = 1.5pt]
       \matrix (dia) [matrix of math nodes, column sep=45pt, row
       sep=40pt]{
H_1L_1(a_1) & V_1R_1(a_1) & V_1L_2(a_2) & H_2L_2(a_2) & V_2R_2(a_2) \\
H_1L_1(b_1) & V_1R_1(b_1) & V_1L_2(b_2) & H_2L_2(b_2) & V_2R_2(b_2)
\\};
\draw[->] (dia-1-1) -- node[above]{$\ve 1(a_1)$} (dia-1-2);
\draw[->] (dia-1-2) -- node[above]{$V_1t)$} (dia-1-3);
\draw[double] (dia-1-3) -- (dia-1-4);
\draw[->] (dia-1-4) -- node[above]{$\ve 2(a_2)$}(dia-1-5);
\draw[->] (dia-2-1) -- node[below]{$\ve 1(b_1)$} (dia-2-2);
\draw[->] (dia-2-2) -- node[below]{$V_1(u)$} (dia-2-3);
\draw[double] (dia-2-3) -- (dia-2-4);
\draw[->] (dia-2-4) -- node[below]{$\ve 2(b_2)$}(dia-2-5);
\draw[->] (dia-1-1) -- node[right]{$H_1L_1m_1$}(dia-2-1);
\draw[->] (dia-1-2) -- node[left]{$V_1R_1m_1$}  (dia-2-2);
\draw[->] (dia-1-3) -- node[left]{$V_1L_2m_2$}  (dia-2-3);
\draw[->] (dia-1-4) -- node[right]{$H_2L_2m_2$}  (dia-2-4);
\draw[->] (dia-1-5) -- node[left]{$V_2R_2m_2$}  (dia-2-5);
\draw[->] (dia-1-1) |- node[pos=0.75,above]{$(\ve 1 \times_T \ve 2)(a_1,t,a_2)$} ($(dia-1-5.north)+(0,0.6)$) -- (dia-1-5); 
\draw[->] (dia-2-1) |- node[pos=0.75,below]{$(\ve 1 \times_T \ve 2)(b_1,u,b_2)$} ($(dia-2-5.south)+(0,-0.6)$) -- (dia-2-5); 
\end{tikzpicture}
\end{center}
commutes in the groupoid $BG$.
\end{enumerate}
\end{defn}

There are functors $\qt cM \to \qt{\bullet}{BG}$,
$M/d \to BG/\bullet$, and $\dqt cMd \to G$ given by
\begin{description}
\item[$\qt cM \to \qt {\bullet}{BG}$] $(\qt c{\ve{}})(s_1,a_1) \xrightarrow{VR(m)} (\qt
  c{\ve{}})(s_2,a_2)$ for $m \in (\qt cM)((s_1,a_1),(s_2,a_2))$
\item[$M/d \to BG/\bullet$] $(\ve{}/d)(a_1,t_1)) \xrightarrow{HL(m)} (\ve{}/d)(a_2,t_2)$ for
  $m \in (M/d)((a_1,t_1),(a_2,t_2))$
\item[$\dqt cMd \to G$]  $(\dqt c{\ve{}}d) (s_1,a_1,t_1) \xrightarrow{e} (\dqt c{\ve{}}d)
  (s_2,a_2,t_2)$ for $m \in (\dqt cMd)((s_1,a_1,t_1),(s_2,a_2,t_2))$
\end{description}
The case of the  two-sided homotopy fibre will be especially pertinent.  
\begin{lemma}\label{lemma:twosidedfibre}
  Let $\dqt cMd$ be the two-sided homotopy fibre of Definition~\ref{defn:funddefn}.\eqref{funddefn:2}.
  \begin{enumerate}
  \item The function $\dqt c{\ve{}}d \colon \Ob{\dqt cMd} \to G$ of
    Definition~\ref{defn:funddefn}.\eqref{funddefn:2} defines a functor from the
    two-sided homotopy fibre 
 to the set $G$, and induces a function
   $\pi_0( \dqt c{\ve{}}d ) \colon \pi_0( \dqt cMd
    ) \to G$ of sets.
  \item For any $g \in G$, the component set for the restricted
    two-sided homotopy fibre is
    \begin{equation*}
      \pi_0((\dqt cMd) \{ \dqt c{\ve{}}d= g \}) = \pi_0( \dqt cMd ) \{ \pi_0(\dqt c{\ve{}}d) = g\}
    \end{equation*}
    and  the full subgroupoid $(\dqt cMd) \{ \dqt c{\ve{}}d= g \}$ is
    the union of the components of $\dqt cMd $ in  $\pi_0( \dqt cMd )
    \{ \pi_0(\dqt c{\ve{}}d) = g\}$.
  \end{enumerate}
\end{lemma}
\begin{proof}
We saw in Definition~\ref{defn:funddefn}.\eqref{funddefn:2} that 
$\dqt c{\ve{}}d$ has the same value on isomorphic objects of $\dqt cMd$.
The second claim follows from the first one. 
\end{proof}

The two-sided homotopy fibre
and the $\ve{}$-restricted two-sided homotopy fibre of Lemma~\ref{lemma:twosidedfibre}
are  Grothendieck constructions
\begin{align*}
  &\qt cM/d = \int_{a \in \mathrm{Ob}(M)} S(c,La) \times T(Ra,d) \\
  &(\qt cM/d)\{\qt{c}{\ve{}}/d=g\} = \int_{a \in \mathrm{Ob}(M)} \{ (s,t) \in S(c,La) \times T(Ra,d) \mid 
   V(t)\ve{}(\alpha)H(s) =g\}
\end{align*}

\begin{rmk}[More on  composable $G$-spans] \label{rmk:clarification}
Consider two composable
$G$-spans as in
Definition~\ref{defn:funddefn}.\eqref{funddefn:4a}. Let $c_1$, $c_2$
be objects of $S_1$, $S_2$. The objects of the groupoid
$\qt{c_1}{(M_1 \times_T M_2)}/c_2$ 
of Definition~\ref{defn:funddefn}.\eqref{funddefn:6}
are all $(s_1,a_1,t,a_2,s_2)$ where
$a_1 \in \mathrm{Ob}(M_1)$, $a_2 \in \mathrm{Ob}(M_2)$,
$t \in T(R_1a_1,L_2a_2)$,
$s_1 \in S_1(c_1,L_1a_1)$, $s_2 \in S_2(R_2a_2,c_2)$:
\begin{center}
   \begin{tikzpicture}[double distance = 1.5pt]
       \matrix (dia) [matrix of math nodes, column sep=25pt, row
       sep=25pt]{
      c_1 & L_1a_1 & R_1a_1 & L_2a_2 & R_2a_2 & c_2 \\};
      \draw[->] (dia-1-1) -- node[above]{$s_1$}(dia-1-2);
      \draw[->] (dia-1-3) -- node[above]{$t$}(dia-1-4);
      \draw[->] (dia-1-5) -- node[above]{$s_2$}(dia-1-6);
   \end{tikzpicture}
\end{center}
The function $\qt{c_1}{(\ve 1 \times_T \ve 2 )}/c_2  \colon \mathrm{Ob}(\qt{c_1}{(M_1 \times_T M_2)}/c_2) \to G$ from
Definition~\ref{defn:funddefn}.\eqref{funddefn:2}
is
\begin{multline*}
  (\qt{c_1}{(\ve 1 \times_T \ve 2 )}/c_2)(s_1,a_1,t,a_2,s_2) = 
V_2(s_2) (\ve 1 \times_T \ve 2) (a_1,t,a_2) H_1(s_1) =
V_2(s_2) \ve 2(a_2) V_1(t) \ve 1(a_1)  H_1(s_1) \\ =
(\ve 2(a_2)/d)(a_2,s_2) V_1(t) (\qt{c_1}{\ve 1})(s_1,a_1) =
((\qt{c_1}{\ve 1}) \times_T  (\ve 2/c_2))(s_1,a_1,t,a_2,s_2)
\end{multline*}
This means that $\dqt{c_1}{(M_1 \times_T M_2)}{c_2} = (\qt{c_1}{M_1})
\times_T (M_2/c_2)$ and $\dqt{c_1}{(\ve 1 \times_T \ve 2)}{c_2} =
(\qt{c_1}{\ve 1})
\times_T (\ve 2/c_2)$.

For fixed object $d$ in $T$, $a_1$ and $a_2$ objects in $M_1$ and
$M_2$ such that $R_1a_1$ and $L_2a_2$ lie in the same 
component of $T$ as $d$,
and $g \in G$, there is a map
\begin{equation}\label{eq:t2t1}
  \coprod_{g_2g_1=g} \big(T(R_1a_1,d)\{\ve 1/d=g_1\} \times T(d,L_2a_2)\{d/\ve 2 = g_2\}\big) \to 
T(R_1a_1,L_2a_2)\{\ve 1\times_T \ve 2 = g\}
\end{equation}
taking $(t_1,t_2)$ to $t_2t_1$. This is because if $t_1 \in T(R_1a_1,d)$ and $t_2 \in T(d,L_2a_2)$ then 
\begin{multline*}
  (\ve 1 \times_T \ve 2)(a_1,t_2t_1,a_2) = \ve 2(a_2) V_1(t_2t_1) \ve 1(a_1)  = 
\ve 2(a_2)V_1(t_2)V_1(t_1)\ve 1(a_1)) \\= \ve 2(a_2) H_2(t_2)  V_1(t_1) \ve 1(a_1) =
 (\qt d{\ve 2})(t_2,a_2)  (\ve 1/d)(a_1,t_1)
\end{multline*}
as $V_1=H_2$. The map \eqref{eq:t2t1} is $|T(d,d)|$-to-$1$ onto since any \m\ $t \in T(R_1a_1,L_2a_2)$ admits
$|T(d,d)|$ factorizations $t=t_2t_1$, $t_1 \in T(R_1a_1,d)$, $t_2 \in T(d,L_2a_2)$, through $d$.
\end{rmk}

\begin{lemma}\label{lemma:main}
For composable $G$-spans of groupoids as in Definition~\ref{defn:funddefn}.\eqref{funddefn:4a},
  \begin{equation*} 
    \chi((M_1 \times_T M_2)\{ \ve{1} \times_T \ve{2} = g\} )=
   \sum_{d \in \pi_0(T)} \sum_{\substack{g_1,g_2 \in G \\  g_2g_1=g}}  \chi((M_1/d)\{\ve 1/d=g_1\}) \chi(T\{d\}) \chi((\qt d{M_2})\{\qt d{\ve 2}=g_2\})
  \end{equation*}
\end{lemma}
\begin{proof}
Let $k_\bullet$ be a co\we\ on $M_1$ and $k^\bullet$ a \we\ on
$M_2$. The \Euc\ of the groupoid
\begin{equation*}
  (M_1 \times_T M_2)\{ \ve 1 \times_T \ve 2 = g\} 
= \int_{(a_1,a_2) \in \Ob{M_1^{\mathrm{op}} \times M_2}} \{t \in T(R_1a_1,L_2a_2) \mid 
\ve 2(a_2)V_1(t) \ve 1(a_1)=g\}
\end{equation*}
is \cite[Definition~1.10, Proposition~2.8]{leinster08}
\begin{align*}
   \sum_{a_1,a_2} &k_{a_1}k^{a_2} |\{ t \in  T(R_1a_1,L_2a_2) \mid \ve 2(a_2)V_1(t) \ve 1(a_1)=g\}| \\=
  &\sum_{d \in \pi_0(T)}\sum_{\substack{a_1,a_2 \\ T(R_1a_1,d) \neq
  \emptyset, T(d,L_2a_2) \neq \emptyset}} k_{a_1}k^{a_2}|\{ t \in  T(R_1a_1,L_2a_2) \mid \ve 2(a_2)V_1(t) \ve 1(a_1)=g\}| \\=
    &   \sum_{d \in \pi_0(T)}  |T(d,d)|^{-1} \sum_{\substack{g_1,g_2 \in G \\ g_2g_1=g}} 
\sum_{\substack{a_1,a_2 \\ T(R_1a_1,d) \neq
  \emptyset, T(d,L_2a_2) \neq \emptyset}}  k_{a_1}k^{a_2} \times \\  & \hspace{5cm}
 |\{ t_1 \in  T(R_1a_1,d) \mid V_1(t_1)\ve 1(a_1) =g_1\}| 
   |\{ t_2 \in  T(d,L_2a_2) \mid  \ve 2(a_2) H_2(t_2) =g_2\}| 
\\=
   &
   \sum_{d \in \pi_0(T)} \chi(T\{d\}) \sum_{\substack{g_1,g_2 \in G \\ g_2g_1=g}} 
\Big( \sum_{a_1} k_{a_1}   |\{ t_1 \in  T(R_1a_1,d) \mid  V_1(t_1)\ve
  1(a_1)=g_1\}| \Big) \times \\
   & \hspace{5cm} \Big(  \sum_{a_2} k^{a_2}  |\{ t_2 \in  T(d,L_2a_2) \mid   \ve 2(a_2) H_2t_2) =g_2\}| \Big)
\\=
    &\sum_{d \in \pi_0(T)} \sum_{g_2g_1=g}  \chi((M_1/d)\{\ve 1/d=g_1\}) \chi(T\{d\}) \chi((\qt d{M_2})\{\qt d{\ve 2}=g_2\})
\end{align*}
The factor $|T(d,d)|^{-1} = \chi(T\{d\})$ appears
at the second equality sign because \eqref{eq:t2t1} is $|T(d,d)|$-to-$1$.
\end{proof}

 \begin{thm}\label{thm:main}
  For composable $G$-spans of groupoids as in Definition~\ref{defn:funddefn}.\eqref{funddefn:4a},
  \begin{equation*}
    [M_1 \times_T M_2, \varepsilon_1 \times_T \varepsilon_2] = [M_1,\varepsilon_1][M_2,\varepsilon_2]
  \end{equation*}
 \end{thm}
 \begin{proof}
   Lemma~\ref{lemma:main} applied to
   \begin{equation*}
((\qt{c_1}{M_1}) \times_T (M_2/c_2))\{(\qt{c_1}{\ve 1}) \times_T (\ve 2/c_2)=g\} =
   (\dqt{c_1}{(M_1 \times_T M_2)}{c_2})\{ \dqt{c_1}{(\ve 1 \times_T \ve 2)}{c_2}=g\}     
   \end{equation*}
gives
   \begin{equation}    
     \begin{aligned}[b]
      \label{eq:chiprodT}
     \chi((\qt{c_1}{(M_1 & \times_T M_2)}/c_2)   \{
     \qt{c_1}{(\varepsilon_1 \times_T \varepsilon_2)}/c_2=g \}) \\ &=
     \sum_{d \in \pi_0(T)}  \sum_{\substack{g_1,g_2 \in G \\  g_2g_1=g}} 
      \chi( (\qt{c_1}{M_1}/d)\{\qt{c_1}{\varepsilon_1}/d = g_1\}) \chi(T\{d\})
      \chi( (\qt{d}{M_2}/c_2)\{ \qt{d}{\varepsilon_2}/c_2 = g_2\}) 
    \end{aligned}
   \end{equation}
   for all $g \in G$, $c_1 \in \Ob{S_1}$, $c_2 \in \Ob{S_2}$. The
   chain of equalities
  \begin{align*}
    ([&M_1,\ve 1][M_2,\ve 2])(c_1,c_2)  \stackrel{\text{Rmk~\ref{rmk:nonabelianG}}}{=} 
    \sum_{d \in \pi_0(T)}     [M_2,\ve 2](d,c_2)  [M_1,\ve 1](c_1,d) 
    \\&=
   \sum_{d \in\pi_0(T)} 
\big(
   \sum_{g_2 \in G} \chi((\dqt{d}{M_2}{c_2})\{\dqt{d}{\ve
     2}{c_2}=g_2\}) \chi(S_2\{c_2\})g_2 \big)
\big(
   \sum_{g_1 \in G} \chi((\dqt{c_1}{M_1}{d})\{\dqt{c_1}{\ve
     1}{d}=g_1\}) \chi(T\{d\})g_1 \big)
    \\
&= 
   \sum_{d \in\pi_0(T)} 
    \sum_{\substack{g_1,g_2 \in G \\ g_2g_1=g}} \big(
    \chi((\dqt{d}{M_2}{c_2})\{\dqt{d}{\ve
     2}{c_2}=g_2\}) \chi(S_2\{c_2\}) 
\chi((\dqt{c_1}{M_1}{d})\{\dqt{c_1}{\ve
     1}{d}=g_1\}) \chi(T\{d\}) \big)g
\\
    &=
    \sum_{d \in\pi_0(T)} \big(\sum_{\substack{g_1,g_2 \in G \\ g_2g_1=g}}
    \chi((\dqt{c_1}{M_1}{d})\{\dqt{c_1}{\ve
     1}{d}=g_1\}) \chi(T\{d\})
 \chi((\dqt{d}{M_2}{c_2})\{\dqt{d}{\ve
     2}{c_2}=g_2\}) \chi(S_2\{c_2\}) \big) g \\
 & =
    \sum_{d \in\pi_0(T)} \big(\sum_{\substack{g_1,g_2 \in G \\ g_2g_1=g}}
    \chi((\dqt{c_1}{M_1}{d})\{\dqt{c_1}{\ve
     1}{d}=g_1\}) \chi(T\{d\})
 \chi((\dqt{d}{M_2}{c_2})\{\dqt{d}{\ve
     2}{c_2}=g_2\}) \chi(S_2\{c_2\}) \big) g \\
 & \stackrel{\text{\eqref{eq:chiprodT}}}{=} \sum_{g \in G} 
     \chi((\qt{c_1}{(M_1 \times_T M_2)}/c_2) \{
     \qt{c_1}{(\varepsilon_1 \times_T \varepsilon_2)}/c_2=g \})
     \chi(S_2\{c_2\}) g \\
    &=
      [M_1 \times_T M_2, \ve 1 \times_T \ve 2](c_1,c_2)
  \end{align*}
now proves the theorem.
 \end{proof}

 \begin{rmk} \label{rmk:nonabelianG}
   Theorem~\ref{thm:main} remains valid also for  non-abelian finite
   groups $G$ provided we use the convention
   \begin{equation*}
      ([M_1,\ve 1][M_2,\ve 2])(c_1,c_2) = 
    \sum_{d \in \pi_0(T)}     [M_2,\ve 2](d,c_2)  [M_1,\ve 1](c_1,d)
   \end{equation*}
  for matrix products over possibly non-commutative group rings $\Q G$.
 \end{rmk}

With the help of a complex character
\func{\rho}{\Q G}{\C}, we can associate a {\em complex\/} valued function
\begin{equation}
  \label{eq:[[M]]}
  [[M,\ve{}]] = \rho \circ [M,\ve{}] \colon \pi_0(S) \times \pi_0(T)
  \xrightarrow{[M,\ve{}]} \Q G \xrightarrow{\rho} \C
\end{equation}
to any $G$-span. These complex matrices still satisfy
$[[M_1 \times_T M_2, \varepsilon_1 \times_T \varepsilon_2]] =
[[M_1,\varepsilon_1]][[M_2,\varepsilon_2]]$ for composable $G$-spans.

\begin{rmk}[Column and row vectors] \label{rmk:columnrow}
To  $G$-spans of the form 
 \begin{center}
      \begin{tikzpicture}[double distance = 1.5pt]
       \matrix (dia) [matrix of math nodes, column sep=65pt, row
       sep=25pt]{
A & \{1\} & B & T \\
S & BG & \{1\} & BG \\};
\draw[->] (dia-1-1) --  (dia-1-2);
\draw[->] (dia-1-3) -- node[above]{$R$}  (dia-1-4);
\draw[->] (dia-2-1) -- node[below]{$H$} (dia-2-2);
\draw[->] (dia-2-3) --  (dia-2-4);
\draw[->] (dia-1-1) -- node[left]{$L$}(dia-2-1);
\draw[->] (dia-1-2) -- (dia-2-2);
\draw[->] (dia-1-3) -- (dia-2-3);
\draw[->] (dia-1-4) -- node[right]{$V$}(dia-2-4);

\path (dia-2-1) -- node[midway]{$\ve 1 \colon HL \implies e $}
(dia-1-2);

\path (dia-2-3) -- node[midway]{$\ve 2 \colon e \implies VR $}
(dia-1-4);

\end{tikzpicture}
\end{center}
 we have associated, respectively, the column and the row vector
 \begin{equation*}
    [A,\ve 1](c,1) = \sum_{g \in G} \chi( (\qt cA)\{ \qt c{\ve 1}=g\})
    g \qquad
  [B,\ve 2](1,d) = \sum_{g \in G} \chi((B/d)\{\ve 2/d = g\})
   T(\{d\})g 
 \end{equation*}
where $c \in \pi_0(S)$, $d \in \pi_0(T)$ and
\begin{gather*}
  \Ob{\qt cA} = \{(a,s) \mid a \in \Ob{A}, s \in S(c,La)\} \quad
  \Ob{(\qt cA)\{\qt c{\ve 1} =g\}} = \{(a,s) \in \Ob{\qt cA} \mid \ve
  1(a)H(s)=g\} \\
  (\qt cA)((a_1,s_1),(a_2,s_2)) =\{m \in A(a_1,a_2) \mid 
  L(m)s_1=s_2\} \\
  \Ob{B/d} = \{(b,t) \mid b \in \Ob{B}, t \in T(Rb,d)\} \quad
  \Ob{ (B/d)\{\ve 2/d =g\}} = \{(b,t) \in \Ob{B/d} \mid V(t) \ve
  2(b)=g\} \\
   (B/d)((b_1,t_1),(b_2,t_2)) = \{ m \in B(b_1,b_2) \mid t_1=
   t_2R(m)\} 
\end{gather*}
\end{rmk}



\section{Categorical aspects of $G$-spans}
\label{sec:functorial-aspects}

According to Theorem~\ref{thm:main}  there are
functors $\mathbf{span}_G \to \mathbf{mat}_R$ for $R=\Q
G$ or $R= \C$
where
\begin{description}
\item[$\mathbf{span}_G$]  is the category
whose objects are finite groupoids $S \xrightarrow{H} BG$ over $BG$
and where 
\m s $(S \xrightarrow{H} BG) \to (T \xrightarrow{V} BG)$ are $G$-spans as in 
  Definition~\ref{defn:funddefn}.\eqref{funddefn:4}.
  Composition of $G$-spans is given by homotopy pull-backs  
 as in
  Definition~\ref{defn:funddefn}.\eqref{funddefn:4a}.
\item[$\mathbf{mat}_R$] is the category whose objects are finite sets
  and \m s
  $I \to J$ are functions $I \times J \to R$.  Composition is matrix
multiplication.
\end{description}

The functor $\mathbf{span}_G \to \mathbf{mat}_R$ takes the groupoid \m\ $S \xrightarrow{H}
BG$ over $BG$ to the set $\pi_0(S)$ and the $G$-span of
Definition~\ref{defn:funddefn}.\eqref{funddefn:4} with apex $M$ and
natural transformation $\ve{} \colon \Ob{M} \to G$ to the matrix
$[M,\ve{}]$ of Definition~\ref{defn:funddefn}.\eqref{funddefn:6} for
$R=\Q G$ or $[[M,\ve{}]]$ for $R=\C$ \eqref{eq:[[M]]}. We have 
commutative diagrams as below where the vertical maps assign matrices to
$G$-spans as in Definition~\ref{defn:funddefn}.\eqref{funddefn:6}
\begin{center}
  \begin{tikzpicture}[double distance = 1.5pt]
       \matrix (dia) [matrix of math nodes, column sep=35pt, row
       sep=25pt]{
    \mathbf{span}_G(S_1 \xrightarrow{H} BG , T \xrightarrow{R} BG) \times 
   \mathbf{span}_G(T \xrightarrow{R} BG, S_2 \xrightarrow{V} BG) &
   \mathbf{span}_G(S_1 \xrightarrow{H} BG,  S_2 \xrightarrow{V} BG) \\
   \mathbf{mat}_R(\pi_0(S),\pi_0(T)) \times \mathbf{mat}_R(\pi_0(T), \pi_0(S_2)) & \mathbf{mat}_R(\pi_0(S_1), \pi_0(S_2)) \\};  
  \draw[->](dia-1-1) -- node[left]{$[\cdot,\cdot] \times [\cdot,\cdot]$} (dia-2-1);
  \draw[->](dia-1-1) -- (dia-1-2);
  \draw[->](dia-2-1) -- (dia-2-2);
  \draw[->](dia-1-2) -- node[right]{$[\cdot,\cdot]$}(dia-2-2);
\end{tikzpicture}
\end{center}
and the upper horizontal map is composition of $G$-spans as in
Definition~\ref{defn:funddefn}.\eqref{funddefn:4a}
 while
the lower one represents matrix multiplication over $R$.


Suppose that  
\begin{center}
      \begin{tikzpicture}[double distance = 1.5pt]
       \matrix (dia) [matrix of math nodes, column sep=30pt, row
       sep=25pt]{
  S &{}& T \\
  & BG \\};
 \draw[->] (dia-1-1) -- node[above]{$\varphi$}(dia-1-3);
\draw[->] (dia-1-1) -- node[left,xshift=-5pt,pos=0.6]{$H$}(dia-2-2);
\draw[->] (dia-1-3) -- node[right,xshift=5pt,pos=0.6]{$V$}(dia-2-2);
 \path (dia-2-2) -- node[pos=.7]{$\ve{}$} (dia-1-2);
      \end{tikzpicture}
\end{center}
is a diagram of finite groupoids and that the function $\ve{} \colon \Ob{S} \to G$ is
a natural transformation $H \implies V\varphi$ such that
\begin{center}
      \begin{tikzpicture}[double distance = 1.5pt]
       \matrix (dia) [matrix of math nodes, column sep=30pt, row
       sep=25pt]{
  Ha_1 & V\varphi a_1 \\
  Ha_2 & V\varphi a_2\\};
 \draw[->] (dia-1-1) -- node[above]{$\ve{}(a_1)$}(dia-1-2);
 \draw[->] (dia-2-1) -- node[below]{$\ve{}(a_2)$}(dia-2-2);
 \draw[->] (dia-1-1) -- node[left]{$H(s)$}(dia-2-1);
\draw[->] (dia-1-2) -- node[right]{$(V\varphi)(s)$}(dia-2-2);
      \end{tikzpicture}
\end{center}
commutes in $G$ for every \m\ $s \colon a_1 \to a_2 \in S$.
Let
$(\varphi,\ve{})_* \in \mathbf{span}_G(S \xrightarrow{H} BG,T
\xrightarrow{V} BG)$ and $(\varphi,\ve{}^{-1})^* \in \mathbf{span}_G(T
\xrightarrow{V} BG, S \xrightarrow{H} BG)$ be the $G$-spans 
\begin{center}
      \begin{tikzpicture}[double distance = 1.5pt]
       \matrix (dia) [matrix of math nodes, column sep=30pt, row
       sep=25pt]{
  S & T \\
  S & BG \\};
 \draw[->] (dia-1-1) -- node[above]{$\varphi$}(dia-1-2);
\draw[->] (dia-2-1) -- node[below]{$H$}(dia-2-2);
\draw[->] (dia-1-2) -- node[right]{$V$}(dia-2-2);
\draw[double] (dia-1-1) -- (dia-2-1);
 \path (dia-2-1) -- node[pos=.5]{$\ve{}(a)$} (dia-1-2);
      \end{tikzpicture}
\qquad \qquad
      \begin{tikzpicture}[double distance = 1.5pt]
       \matrix (dia) [matrix of math nodes, column sep=30pt, row
       sep=25pt]{
  S & S \\
  T & BG \\};
\draw[->] (dia-1-1) -- node[left]{$\varphi$}(dia-2-1);
\draw[->] (dia-1-2) -- node[right]{$H$}(dia-2-2);
\draw[->] (dia-2-1) -- node[below]{$V$}(dia-2-2);
\draw[double] (dia-1-1) -- (dia-1-2);
 \path (dia-2-1) -- node[pos=.5]{$\ve{}(a)^{-1}$} (dia-1-2);
      \end{tikzpicture}
\end{center}
and let $T(\varphi(c),d) \xrightarrow{\varphi_*(c,d)} G
\xleftarrow{\varphi^*(c,d)} T(d,\varphi(c))$ be the maps
$\varphi_*(c,d)(t)=V(t)\ve{}(c)$,
$\varphi^*(d,c)(t)=\ve{}(c)^{-1}V(t)$, $c \in \pi_0(S)$,  $d \in \pi_0(T)$.  
\begin{prop}\label{prop:phi*}
 The functions $[(\varphi, \ve{})_*] \colon \pi_0(S) \times
 \pi_0(T) \to \Q G$ and $[(\varphi, \ve{}^{-1})^*] \colon \pi_0(T) \times
 \pi_0(S) \to \Q G$ are
  \begin{equation*}
    [(\varphi, \ve{})_*](c,d)=\chi(T\{d\})\sum_{g \in G}
    |\varphi_*(c,d)^{-1}(g)|g \qquad
    [(\varphi, \ve{}^{-1})^*](d,c)=\chi(S\{c\})\sum_{g \in G}
    |\varphi^*(d,c)^{-1}(g)|g
  \end{equation*}
\end{prop}
\begin{proof}
Let $c$ be an object of $S$ and $d$ an object of $T$.  In
$(\varphi,\ve{})_*$, the objects of
$\dqt cSd$ are triples $(s,a,t)$ where $a \in \Ob{S}$, $s \in S(c,a)$,
$t \in T(\varphi(a),d)$. A \m\ $(s_1,a_1,t_1) \to (s_2,a_2,t_2)$ in
$\dqt cSd$ is a \m\ $m \in S(a_1,a_2)$ such that $s_2=ms_1$ in $S$ and
$t_1=t_2\varphi(m)$ in $T$. We note that $\dqt cSd$ is equivalent to
its component set since there is at most one \m\ from one object to
another. Indeed, $(s_1,a_1,t_1)$ and  $(s_2,a_2,t_2)$ are isomorphic
if and only if $t_1\varphi(s_1)=t_2\varphi(s_2)$. The diagram
 \begin{center}
    \begin{tikzpicture}[double distance = 1.5pt]
       \matrix (dia) [matrix of math nodes, column sep=30pt, row
       sep=40pt]{
      \pi_0( \dqt cSd) && T(\varphi(c),d) \\
      & G \\};
\draw[->] (dia-1-1) -- node[above]{$(s,a,t) \to t \varphi(s)$} node[below]{$\simeq$}(dia-1-3); 
\draw[->] (dia-1-1) -- node[below,sloped]{$\dqt c{\ve{}}d$}(dia-2-2);
\draw[->] (dia-1-3) -- node[below,sloped]{$ V(t) \ve{}(c) \leftarrow t
  $}(dia-2-2);
\end{tikzpicture}
\end{center}
commutes since $(\dqt c{\ve{{}}}d)(s,a,t) = V(t) \ve{}(a) H(s) = V(t)
  V(\varphi(s)) \ve{}(c) = V(t \varphi(s)) \ve{}(c)$. This shows that
  $\pi_0(\dqt cSd) \{\dqt c{\ve{}}d = g\} = \varphi_*(c,d)^{-1}(g)$.

The proof for $(\varphi,\ve{})^*$ is similar. The crucial point is that
the diagram
  \begin{center}
    \begin{tikzpicture}[double distance = 1.5pt]
       \matrix (dia) [matrix of math nodes, column sep=30pt, row
       sep=40pt]{
      \pi_0( \dqt dSc) && T(\varphi(c),d) \\
      & G \\};
\draw[->] (dia-1-1) -- node[above]{$(t,a,s) \to \varphi(s)t$} node[below]{$\simeq$}(dia-1-3); 
\draw[->] (dia-1-1) -- node[below,sloped]{$\dqt d{\ve{}}c$}(dia-2-2);
\draw[->] (dia-1-3) -- node[below,sloped]{$ \ve{}(c)^{-1}  V(t) \leftarrow t
  $}(dia-2-2);
\end{tikzpicture}
\end{center}
commutes since $(\dqt d{\ve{}}c)(t,a,s) = H(s) \ve{}(a)^{-1} V(t) =
\ve{}(c)^{-1}V(\varphi(s))V(t) = \ve{}(c)^{-1}V(\varphi(s)t)$. This
shows that
$\pi_0(\dqt dSc)\{ \dqt d{\ve{}}c = g\} = \varphi^*(d,c)^{-1}(g)$.
\end{proof}


For any finite subgroup $U$ of $G$, the average sum in $\Q G$ of its elements
\begin{equation}
  \label{eq:defnUbar}
  \overline{U}= \frac{1}{|U|} \sum_{g \in U} g
\end{equation}
is an idempotent, $\overline{U}^2 = \overline{U}$, in 
$\Q G$.

Consider the $G$-spans of Figure~\ref{fig:Gspans} where $H$, $L$, $R$,
$V$ are groupoid \m s and $e \colon \Ob{S} \to G$,
$e \colon \Ob{T} \to G$, are constant functions onto the identity
element $e$ of $G$. We write $[M,\ve{}]$ for the matrix of the middle,
general, $G$-span. The two other $G$-spans are
$(\mathrm{id}_S,e)_*=(\mathrm{id}_S,e)^*$ and
$(\mathrm{id}_T,e)_*=(\mathrm{id}_T,e)^*$ defined by identity \m s
over $BG$.

\begin{figure}[t]
  \centering
    \begin{tikzpicture}[double distance = 1.5pt]
       \matrix (dia) [matrix of math nodes, column sep=45pt, row
       sep=25pt]{
  S & S & M & T & T & T \\
  S & BG & S & BG & T &BG\\};
\draw[double] (dia-1-1) -- (dia-1-2);
\draw[double] (dia-1-1) -- (dia-2-1);
\draw[->] (dia-1-2) -- node[right]{$H$}(dia-2-2);
\draw[->] (dia-2-1) -- node[below]{$H$} (dia-2-2);
\path (dia-2-1) -- node[midway]{$H \xRightarrow{e} H$} (dia-1-2); 
\draw[->](dia-1-3) -- node[above]{$R$}(dia-1-4);
\draw[->](dia-2-3) -- node[below]{$H$}(dia-2-4);
\draw[->](dia-1-3) --node[left]{$L$} (dia-2-3);
\draw[->](dia-1-4) -- node[right]{$V$}(dia-2-4);
\path (dia-2-3) -- node[midway]{$HL \xRightarrow{\ve{}} VR$} (dia-1-4);
\draw[double] (dia-1-5) -- (dia-2-5);
\draw[double] (dia-1-5) -- (dia-1-6);
\draw[->] (dia-2-5) -- node[below]{$V$} (dia-2-6);
\draw[->] (dia-1-6) -- node[right]{$V$} (dia-2-6);
\path (dia-2-5) -- node[midway]{$V \xRightarrow{e} V$} (dia-1-6);
 \end{tikzpicture}
  \caption{Three $G$-spans}
\label{fig:Gspans}
\end{figure}

\begin{lemma} \label{lemma:idempotentmatrix}
  $[(\mathrm{id}_S,e)_*]$ is an idempotent diagonal matrix over $\Q G$ with
  diagonal entries
  \begin{equation*}
    [(\mathrm{id}_S,e)_*] (c,c)  = \overline{HS(c,c)}, \qquad c \in \pi_0(S)
  \end{equation*}
\end{lemma}
\begin{proof}
  This is the matrix of the $G$-span $(\mathrm{id}_S,e)_*$ of
  Proposition~\ref{prop:phi*} which in the diagonal has
\begin{equation*}
  [(\mathrm{id}_S,e)_*](c,c) = \sum_{g \in G} \frac{1}{|S(c,c)|} |\{
  s \in S(c,c) \mid H(s) = g \}| =
 \sum_{g \in HS(c,c)} \frac{1}{|HS(c,c)|}g = \overline{HS(c,c)}
\end{equation*}
while the off-diagonal elements are $0$.
\end{proof}



\begin{rmk}\label{rmk:diagonalidempotent}
  The diagonal idempotent matrix of
  Lemma~\ref{lemma:idempotentmatrix} and its complex companion
\begin{equation*}
  [(\mathrm{id}_S,e)_*] = \mathrm{diag}(\overline{HS(c,c)})_{c \in \pi_0(S)} 
\qquad
  [[(\mathrm{id}_S,e)_*]] = \rho([(\mathrm{id}_S,e)_*])
\end{equation*}
 may not be identity matrices unless $HS(c,c)$ is the
trivial group for all objects $c$ of $S$. For instance, if
$G \leq U(1)$ is a finite cyclic group of complex units,
\begin{equation*}
  [[(\mathrm{id}_S,e)_*]](c,c) = \rho( \overline{HS(c,c)}) =
  \begin{cases}
    1 & |HS(c,c)| =1 \\
    0 & |HS(c,c)| > 1 
  \end{cases} \qquad c \in \pi_0(S)
\end{equation*}
so $[[(\mathrm{id}_S,e)_*,e]]$ is a diagonal matrix with $0$s or $1$s in the diagonal.
\end{rmk}


Consider two $G$-spans
 \begin{center} 
  \begin{tikzpicture}
         \matrix (dia) [matrix of math nodes, column sep=60pt, row sep=40pt]{
      M_1 & T\\
      S & BG \\};
    \draw[->] (dia-1-1) --  node[above]{$R_1$} (dia-1-2);
     \draw[->] (dia-1-1) --  node[left]{$L_1$} (dia-2-1);
      \draw[->] (dia-1-2) --  node(V)[right]{$V$} (dia-2-2);
      \draw[->] (dia-2-1) --  node[below]{$H$} (dia-2-2);
      \path (dia-2-1) -- node[midway]{$HL_1\stackrel{\ve 1}{\implies}VR_1$}(dia-1-2); 
  \end{tikzpicture}
  \qquad
\begin{tikzpicture}
         \matrix (dia) [matrix of math nodes, column sep=60pt, row sep=40pt]{
      M_2 & T\\
      S & BG \\};
    \draw[->] (dia-1-1) --  node[above]{$R_2$} (dia-1-2);
     \draw[->] (dia-1-1) --  node[left]{$L_2$} (dia-2-1);
      \draw[->] (dia-1-2) --  node(V)[right]{$V$} (dia-2-2);
      \draw[->] (dia-2-1) --  node[below]{$H$} (dia-2-2);
      \path (dia-2-1) -- node[midway]{$HL_2\stackrel{\ve 2}{\implies}VR_2$}(dia-1-2); 
  \end{tikzpicture}
  \end{center}
from $S \xrightarrow{H} BG$ to $T \xrightarrow{V} BG$.

\begin{defn}\label{defn:2cell}
  A $G$-span \m\  from the $G$-span $M_1$ to the $G$-span $M_2$ is a triple $(A,\Phi,B)$ of a functor
  $\Phi \colon M_1 \to M_2$ and two natural transformations
  $A \colon L_1 \implies L_2 \Phi$, $B \colon R_1 \implies R_2\Phi$
  such that $V(Bx)\ve 1(x) = \ve 2(\Phi x)H(Ax)$ for all
  $x \in \Ob{M_1}$.
\end{defn}

The data of the $G$-span \m\  $(A,\Phi,B) \colon M_1 \implies M_2$, sometimes graphically presented as 
 \begin{center}
      \begin{tikzpicture}[double distance = 1.5pt]
       \matrix (dia) [matrix of math nodes, column sep=40pt, row
       sep=25pt]{
S \xrightarrow{H} BG & T \xrightarrow{V} BG \\};
 \draw[->] (dia-1-1) to[bend left=45] node[above]{$M_1$} (dia-1-2);
  \draw[->] (dia-1-1) to[bend right=45]
  node[below]{$M_2$}(dia-1-2);
   \node[yshift=15pt]  (A) at ($(dia-1-1)!0.5!(dia-1-2)$) {};
  \node[yshift=-15pt]  (B) at ($(dia-1-1)!0.5!(dia-1-2)$) {};
   \draw[double, -implies] (A) -- node[right]{$\Phi$} (B);
\end{tikzpicture}
\end{center}
inhabit the diagram 
\begin{center}
      \begin{tikzpicture}[double distance = 1.5pt]
       \matrix (dia) [matrix of math nodes, column sep=40pt, row
       sep=25pt]{
& S &&& S \\
BG && M_1 & M_2 && BG \\
& T &&& T \\};
\draw[double] (dia-1-2) -- (dia-1-5);
\draw[double] (dia-3-2) -- (dia-3-5);

\draw[->](dia-2-3) --node[above,sloped,pos=0.3]{$L_1$} (dia-1-2);
\draw[->](dia-2-3) --node[below,sloped,pos=0.3]{$R_1$} (dia-3-2);

\draw[->] (dia-1-2) --node[above,sloped]{$H$} (dia-2-1);
\draw[->] (dia-3-2) --node[below,sloped]{$V$} (dia-2-1);

\path (dia-2-1) -- node[midway]{$HL_1\stackrel{\ve 1}{\implies}VR_1$}(dia-2-3);
\draw[->](dia-2-4) --node[above,sloped,pos=0.3]{$L_2$} (dia-1-5);
\draw[->](dia-2-4) --node[below,sloped,pos=0.3]{$R_2$} (dia-3-5);

\draw[->] (dia-1-5) --node[above,sloped]{$H$} (dia-2-6);
\draw[->] (dia-3-5) --node[below,sloped]{$V$} (dia-2-6);

\path (dia-2-4) -- node[midway]{$HL_2\stackrel{\ve 2}{\implies}VR_2$}(dia-2-6);

\draw[->] (dia-2-3) --node[above]{$\Phi$} (dia-2-4);
\node[yshift=25pt]  at ($(dia-2-3)!0.5!(dia-2-4)$) {$L_1\stackrel{A}{\implies}L_2\Phi$};
\node[yshift=-25pt]  at ($(dia-2-3)!0.5!(dia-2-4)$) {$R_1\stackrel{B}{\implies}R_2\Phi$};
\end{tikzpicture}
\end{center}
where 
\begin{center}
      \begin{tikzpicture}[double distance = 1.5pt]
       \matrix (dia) [matrix of math nodes, column sep=30pt, row
       sep=25pt]{
   HL_1 & VR_1 \\ HL_2\Phi & VR_2\Phi \\};
  \draw[double,-implies] (dia-1-1) -- node[above]{$\ve 1$} (dia-1-2);
  \draw[double,-implies] (dia-2-1) -- node[below]{$\ve 2$} (dia-2-2);
  \draw[double,-implies] (dia-1-1) -- node[left]{$HA$} (dia-2-1);
  \draw[double,-implies] (dia-1-2) -- node[right]{$VB$} (dia-2-2);
\end{tikzpicture}
\end{center}
is a commutative diagram of natural transformations between groupoid
\m s
from $M_1$ to $BG$.

The functor $\Phi \colon M_1 \to M_2$ induces a functor
$\dqt c{\Phi}d \colon \dqt c{M_1}d \to \dqt c{M_2}d$ of two-sided
homotopy fibres over $c \in \pi_0(S)$ and $d \in \pi_0(T)$ with the
effect
\begin{equation*}
  (c \xrightarrow{\alpha} L_1x, x, R_1x \xrightarrow{\beta} d) \to
  (c \xrightarrow{\alpha} L_1x \xrightarrow{Ax} L_2\Phi x, \Phi x, 
  R_2\Phi x \xrightarrow{(Bx)^{-1}} R_1x \xrightarrow{\beta} d) 
\end{equation*}
on objects. Because
\begin{multline*}
  (\dqt c{\ve 2}d) (c \xrightarrow{\alpha} L_1x \xrightarrow{Ax} L_2\Phi x, \Phi x, 
  R_2\Phi x \xrightarrow{(Bx)^{-1}} R_1x \xrightarrow{\beta} d)\\ =
 V(\beta \circ  (Bx)^{-1}) \ve 2(\Phi x) H(Ax \circ \alpha) 
  = V(\beta) V(Bx)^{-1} \ve 2(\Phi x) H(Ax) H(\alpha) =
 V(\beta) \ve 1(x) H(\alpha) \\=
(\dqt c{\ve 1}d)(c \xrightarrow{\alpha} L_1x, x, R_1x \xrightarrow{\beta} d) 
\end{multline*}
the functor $\dqt c{\Phi}d$ retricts to
functors of retricted two-sided homotopy fibres
\begin{equation*}
  (\dqt c{M_1}d)\{ \dqt c{\ve 1}d =g \} \xrightarrow{\dqt c{\Phi}d}
  (\dqt
  c{M_2}d)\{ \dqt c{\ve 2}d =g \}
\end{equation*}
for all $g \in G$.

The vertical composition of the $G$-span \m s  $
M_1 \stackrel{(A_1,\Phi_1,B_1)}{\implies} M_2
\stackrel{(A_2,\Phi_2,B_2)}{\implies} M_3$ between the $G$-spans $
(S \xrightarrow{H} BG) \xrightarrow{M_1,M_2,M_3} (T \xrightarrow{V} BG)$
is defined to be the $G$-span \m\
\begin{equation}\label{eq:vertical}
  (A_2,\Phi_2,B_2) \ast (A_1,\Phi_1,B_1) =
  (A_2\Phi_1 \circ
A_1, \Phi_2 \circ \Phi_1, B_2 \Phi_1 \circ B_1) \colon M_1 \implies M_3
\end{equation}
This is indeed a
$G$-span \m\ from $M_1$ to $M_3$ because
\begin{multline*}
  V(B_2\Phi_1(x) \circ B_1x) \ve 1(x) = V(B_2\Phi_1(x))V(B_1x)\ve 1(x) =
  V(B_2\Phi_1(x)) \ve 2(\Phi_1x) H(A_1x) \\ = \ve
  3(\Phi_2\Phi_1x)H(A_2(\Phi_1x))H(A_1x)
 =\ve
  3(\Phi_2\Phi_1x)H(A_2(\Phi_1x) \circ A_1x)
\end{multline*}
for all $x \in \Ob{M_1}$. The diagram, from which we may extract the composable natural
transformations $L_1\stackrel{A_1}{\implies}L_2\Phi_1\stackrel{A_2\Phi_1}{\implies}L_3\Phi_2\Phi_1$
and $R_1\stackrel{B_1}{\implies}R_2\Phi_1\stackrel{B_2\Phi_1}{\implies}R_3\Phi_2\Phi_1$,
\begin{center}
      \begin{tikzpicture}[double distance = 1.5pt]
       \matrix (dia) [matrix of math nodes, column sep=90pt, row
       sep=35pt]{
& S & S & S \\
BG & M_1 & M_2 & M_3 & BG \\
& T & T & T \\};
\draw[double] (dia-1-2) -- (dia-1-3);
\draw[double] (dia-1-3) -- (dia-1-4);
\draw[double] (dia-3-2) -- (dia-3-3);
\draw[double] (dia-3-3) -- (dia-3-4);
\draw[->] (dia-2-2) -- node[left]{$L_1$} (dia-1-2);
\draw[->] (dia-2-3) -- node[left]{$L_2$} (dia-1-3);
\draw[->] (dia-2-4) -- node[right]{$L_3$} (dia-1-4);
\draw[->] (dia-2-2) -- node[left]{$R_1$} (dia-3-2);
\draw[->] (dia-2-3) -- node[left]{$R_2$} (dia-3-3);
\draw[->] (dia-2-4) -- node[right]{$R_3$} (dia-3-4);

\draw[->] (dia-2-2) --node[above]{$\Phi_1$} (dia-2-3);
\draw[->] (dia-2-3) --node[above]{$\Phi_2$} (dia-2-4);

\path (dia-2-2) -- node[midway]{$L_1 \stackrel{A_1}{\implies} L_2\Phi_1$}(dia-1-3);
\path (dia-2-3) -- node[midway]{$L_2 \stackrel{A_2}{\implies} L_3\Phi_2$}(dia-1-4);
\path (dia-2-2) -- node[midway]{$R_1 \stackrel{B_1}{\implies} R_2\Phi_1$}(dia-3-3);
\path (dia-2-3) -- node[midway]{$R_2\stackrel{B_2}{\implies}R_3\Phi_2$}(dia-3-4);
\path (dia-2-1) -- node[midway]{$HL_1 \stackrel{\ve 1}{\implies} VR_1$}(dia-2-2);
\path (dia-2-4) -- node[midway]{$HL_3\stackrel{\ve 3}{\implies}VR_3$}(dia-2-5);

\draw[->] (dia-1-2) -- node[above,sloped]{$H$}(dia-2-1);
\draw[->] (dia-3-2) -- node[below,sloped]{$V$}(dia-2-1);

\draw[->] (dia-1-4) -- node[above,sloped]{$H$}(dia-2-5);
\draw[->] (dia-3-4) -- node[below,sloped]{$V$}(dia-2-5);
\end{tikzpicture}
\end{center}
exhibits the ingredients of the vertical composition of the two $G$-span \m s. (The natural
transformation $HL_2 \stackrel{\ve 2}{\implies} VR_2$ of the $G$-span $M_2$ has been suppressed.)


\begin{exmp}
  Any $G$-span $M \in \mathbf{span}_G(S \xrightarrow{H} BG,  T
\xrightarrow{V} BG)$ admits an obvious $G$-span \m\ $(e,\Phi,e)$ to
the homotopy pull-back $S \times_{BG} T \in  \mathbf{span}_G(S \xrightarrow{H} BG,  T
\xrightarrow{V} BG)$
\begin{center}
      \begin{tikzpicture}[double distance = 1.5pt]
       \matrix (dia) [matrix of math nodes, column sep=30pt, row
       sep=25pt]{
& S &&& S \\
BG && M & S \times_{BG} T && BG \\
& T &&& T \\};
\draw[double] (dia-1-2) -- (dia-1-5);
\draw[double] (dia-3-2) -- (dia-3-5);

\draw[->](dia-2-3) --node[above,sloped]{$L$} (dia-1-2);
\draw[->](dia-2-3) --node[below,sloped]{$R$} (dia-3-2);

\draw[->] (dia-1-2) --node[above,sloped]{$H$} (dia-2-1);
\draw[->] (dia-3-2) --node[below,sloped]{$V$} (dia-2-1);

\path (dia-2-1) -- node[midway]{$HL\stackrel{\ve{}}{\implies}VR$}(dia-2-3);
\draw[->](dia-2-4) --node[above,sloped]{$p_1$} (dia-1-5);
\draw[->](dia-2-4) --node[below,sloped]{$p_2$} (dia-3-5);

\draw[->] (dia-1-5) --node[above,sloped]{$H$} (dia-2-6);
\draw[->] (dia-3-5) --node[below,sloped]{$V$} (dia-2-6);

\path (dia-2-4) -- node[midway]{$(x,g,y) \to g$}(dia-2-6);

\draw[->] (dia-2-3) --node[above]{$\Phi$} (dia-2-4);
\node[yshift=25pt]  at ($(dia-2-3)!0.5!(dia-2-4)$) {$L\stackrel{e}{=}p_1\Phi$};
\node[yshift=-25pt]  at ($(dia-2-3)!0.5!(dia-2-4)$) {$R\stackrel{e}{=}p_2\Phi$};
\end{tikzpicture}
\end{center}
where $\Phi(x)=(Lx,\ve{}x,Rx)$ for all $x \in \Ob{M}$. In
the homotopy pull-back,  $\Ob{S \times_{BG} T} =\Ob{S} \times G
\times \Ob{T}$ and $p_1(x,g,y)=x$, $p_2(x,g,y)=y$.
\end{exmp}

\begin{exmp}
 There are $G$-span \m s
\begin{center}
      \begin{tikzpicture}[double distance = 1.5pt]
       \matrix (dia) [matrix of math nodes, column sep=40pt, row
       sep=25pt]{
S \xrightarrow{H} BG & S \xrightarrow{H} BG \\};
 \draw[->] (dia-1-1) to[bend left=45] node[above]{$(\mathrm{id}_S,e)_*$} (dia-1-2);
  \draw[->] (dia-1-1) to[bend right=45]
  node[below]{$(\varphi,\ve{})_* \times_T
    (\varphi,\ve{}^{-1})^*$}(dia-1-2);
   \node[yshift=15pt]  (A) at ($(dia-1-1)!0.5!(dia-1-2)$) {};
  \node[yshift=-15pt]  (B) at ($(dia-1-1)!0.5!(dia-1-2)$) {};
   \draw[double, -implies] (A) -- (B);
\end{tikzpicture}
\qquad
      \begin{tikzpicture}[double distance = 1.5pt]
       \matrix (dia) [matrix of math nodes, column sep=40pt, row
       sep=25pt]{
T \xrightarrow{V} BG & T \xrightarrow{V} BG \\};
 \draw[->] (dia-1-1) to[bend left=45] node[above]{$(\mathrm{id}_T,e)_*$} (dia-1-2);
  \draw[->] (dia-1-1) to[bend right=45]
  node[below]{$
    (\varphi,\ve{}^{-1})^* \times_S (\varphi,\ve{})_* $}(dia-1-2);
   \node[yshift=15pt]  (A) at ($(dia-1-1)!0.5!(dia-1-2)$) {};
  \node[yshift=-15pt]  (B) at ($(dia-1-1)!0.5!(dia-1-2)$) {};
   \draw[double, -implies] (B) -- (A);
\end{tikzpicture}
\end{center}
based on the functors $S \to S \times_T S \colon x \to
(x,\mathrm{id}_{\varphi x}, x)$, $S \times_S S \to T \colon
(x,s,x) \to \varphi x$. It is essential that $(\ve{} \times_T \ve{}^{-1})
(x,\mathrm{id}_{\varphi x}, x) =\ve{}(x)^{-1} \ve{}(x)= e$ and
$(\ve{}^{-1} \times_S \ve{})(x,s,x) =
\ve{}(x)H(s)\ve{}(x)^{-1}=V(\varphi x)$. 
\end{exmp}


The matrix $[M,\ve{}]$ of any $G$-span $M \in \mathbf{span}_G(S
\xrightarrow{H} BG, T \xrightarrow{V} BG)$
 will always lie in the $(+1)$-eigenspaces of
left-multiplication with $[(\mathrm{id}_S,e)_*]$ and
right-multiplication with $[(\mathrm{id}_T,e)_*]$.

\begin{prop}\label{prop:restrictM}
The matrices of the $G$-spans of Figure~\ref{fig:Gspans} satisfy the equations
\begin{equation*}
[(\mathrm{id}_S,e)_*] [M,\ve{}] = [M,\ve{}]  =  [M,\ve{}] [(\mathrm{id}_T,e)_*] 
\end{equation*}
\end{prop}
\begin{proof}

The composition  of the $G$-spans $(\mathrm{id}_S,e)_*$ and $M$ is
the $G$-span
\begin{center}
      \begin{tikzpicture}[double distance = 1.5pt]
       \matrix (dia) [matrix of math nodes, column sep=30pt, row
       sep=25pt]{
S \times_S M & T \\
S & BG \\};
\draw[->](dia-1-1) -- node[above]{$Rp_2$}(dia-1-2);
\draw[->](dia-2-1) -- node[below]{$H$}(dia-2-2);

\draw[->](dia-1-1) -- node[left]{$p_1$}(dia-2-1);
\draw[->](dia-1-2) -- node[right]{$V$}(dia-2-2);

\path (dia-2-1) --node[midway]{$e \times_S \ve{}$} (dia-1-2);
\end{tikzpicture}
\end{center}
derived from the diagram
\begin{center}
      \begin{tikzpicture}[double distance = 1.5pt]
       \matrix (dia) [matrix of math nodes, column sep=30pt, row
       sep=25pt]{
S \times_S M & M & T \\
S & S & BG \\
S & BG \\};
\draw[->](dia-1-1) -- node[above]{$p_2$} node[below]{$\simeq$}(dia-1-2);
\draw[->](dia-1-2) -- node[above]{$R$} (dia-1-3);
\draw[double](dia-2-1) -- (dia-2-2);
\draw[->](dia-2-2) -- node[below]{$H$} (dia-2-3);
\draw[->](dia-3-1) -- node[below]{$H$}(dia-3-2);

\draw[->] (dia-1-1) --node[left]{$p_1$} (dia-2-1);
\draw[double] (dia-2-1) --(dia-3-1);
\draw[->] (dia-1-2) --node[left]{$L$} (dia-2-2);
\draw[->] (dia-2-2) --node[right]{$H$} (dia-3-2);
\draw[->] (dia-1-3) --node[right]{$V$} (dia-2-3);

\path (dia-3-1) --node[midway]{$e$} (dia-2-2);
\path (dia-2-2) --node[midway]{$\ve{}$} (dia-1-3);
\end{tikzpicture}
\end{center}
As $[(\mathrm{id}_S,e)_*][M,\ve {}] = [S \times_S M, e \times_S \ve{}]$ by
Theorem~\ref{thm:main}, it remains to show that $[S \times_S M, e
\times_S \ve{}] = [M,\ve{}]$.
The objects of $S \times_S M$ are $(x_1,s,x_2)$ for $x_1 \in \Ob{S}$,
$x_2 \in \Ob{M}$, $s \in S(x_1,Lx_2)$ and $(e \times_s
\ve{})(x_1,s,x_2) = \ve{}(x_2)H(s)$ by Definition~\ref{defn:funddefn}.\eqref{funddefn:4a}. 

The functor $p_2 \colon S \times_S M \to M$, which takes the object $(x_1,s,x_2)$
to $x_2$, is an equivalence: There is a functor
$q \colon M \to S \times_S M$, which on objects is the map
$x \to (Lx,\mathrm{id}_{Lx},x)$, and the two compositions are
isomorphic to identity functors. 
Both $q$ and $p_2$ can easily be
promoted to $G$-span \m s between $S \times_S M$ and $M$. These \m s
produce equivalences between  the restricted two-sided homotopy fibres
$(\dqt c{(S \times_S M)}d)\{\dqt c{\ve{}}d =g \}$
and
$(\dqt cMd) \{ \dqt c{(e \times_S \ve{})}d = g \}$.
As equivalent groupoids
have identical \Euc s \cite[Proposition~2.4]{leinster08},
$[S \times_S M, e
\times_S \ve{}] = [M,\ve{}]$ follows.
\end{proof}

We stress that 
Proposition~\ref{prop:restrictM}, which is equivalent to
\begin{align*}
   & [M,\ve{}](c,d) \in \overline{VT(d,d)} \cdot \Q G \cdot
  \overline{HS(c,c)} 
 \\
    &  [[M,\ve{}]](c,d) \in \rho(\overline{VT(d,d)}) \cdot \C \cdot
   \rho(\overline{HS(c,c)})
\end{align*}
for all $c \in \pi_0(S)$,  $d \in \pi_0(T)$,
restricts the possible values of the function
$[M,\ve{}]$.  For instance, if $G=C_6 \leq U(1)$ is cyclic of order
$6$ and $HS(c,c)=C_2$, $VT(d,d)=C_3$ for some $c \in \pi_0(S)$,
$d \in \pi_0(T)$, then
$[M,\ve{}](c,d) \in \Q  \cdot \overline{C_6}$ as
$\overline{C_2}\overline{C_3}=\overline{C_6} = \frac{1}{6}\sum_{i=0}^5
X^i$ in $\Q C_6 = \Z[X]/(X^6-1)$.

The consequence is even more drastic for the associated complex
matrices of \eqref{eq:[[M]]} when $G \leq U(1)$ as then
$\rho(\overline{HS(c,c)})=0$ ($\rho(\overline{VT(d,d)})=0$) 
for any non-trivial $HS(c,c)$ ($VT(d,d)$).

\begin{cor}\label{cor:drastic}
If $G
\leq U(1)$ is a finite group of complex roots of unity, then
$[[M,\ve{}]](c,d)=0$ unless both $HS(c,c)$ and $VT(d,d)$ are trivial groups.
\end{cor}


\begin{rmk}[The strict $2$-category $\mathbf{span}_G$]
  Suppose that

 \begin{center}
      \begin{tikzpicture}[double distance = 1.5pt]
       \matrix (dia) [matrix of math nodes, column sep=30pt, row
       sep=25pt]{
     & M_2 & U \\
     M_1 & T & BG \\
     S & BG \\};
  \draw[->] (dia-1-2) -- node[above]{$R_2$}(dia-1-3);
 \draw[->] (dia-2-1) -- node[above]{$R_1$}(dia-2-2);
 \draw[->] (dia-2-2) -- node[below]{$V_1$}(dia-2-3);
 \draw[->] (dia-3-1) -- node[below]{$H_1$}(dia-3-2);

\draw[->] (dia-1-3) -- node[right]{$V_2$}(dia-2-3);
 \draw[->] (dia-1-2) -- node[left]{$L_2$}(dia-2-2);
\draw[->] (dia-2-1) -- node[left]{$L_1$}(dia-3-1);
\draw[->] (dia-2-2) -- node[right]{$V_1$}(dia-3-2);

\path (dia-3-1) -- node[midway]{$\ve 1$}(dia-2-2);
\path (dia-2-2) -- node[midway]{$\ve 2$} (dia-1-3);
\end{tikzpicture}
\qquad
 \begin{tikzpicture}[double distance = 1.5pt]
       \matrix (dia) [matrix of math nodes, column sep=30pt, row
       sep=25pt]{
     & M_2' & U \\
     M_1' & T & BG \\
     S & BG \\};
  \draw[->] (dia-1-2) -- node[above]{$R_2'$}(dia-1-3);
 \draw[->] (dia-2-1) -- node[above]{$R_1'$}(dia-2-2);
 \draw[->] (dia-2-2) -- node[below]{$V_1$}(dia-2-3);
 \draw[->] (dia-3-1) -- node[below]{$H_1$}(dia-3-2);

\draw[->] (dia-1-3) -- node[right]{$V_2$}(dia-2-3);
 \draw[->] (dia-1-2) -- node[left]{$L_2'$}(dia-2-2);
\draw[->] (dia-2-1) -- node[left]{$L_1'$}(dia-3-1);
\draw[->] (dia-2-2) -- node[right]{$V_1$}(dia-3-2);

\path (dia-3-1) -- node[midway]{$\ve 1'$}(dia-2-2);
\path (dia-2-2) -- node[midway]{$\ve 2'$} (dia-1-3);
\end{tikzpicture}
\qquad
 \begin{tikzpicture}[double distance = 1.5pt]
       \matrix (dia) [matrix of math nodes, column sep=30pt, row
       sep=25pt]{
     & M_2'' & U \\
     M_1'' & T & BG \\
     S & BG \\};
  \draw[->] (dia-1-2) -- node[above]{$R_2''$}(dia-1-3);
 \draw[->] (dia-2-1) -- node[above]{$R_1''$}(dia-2-2);
 \draw[->] (dia-2-2) -- node[below]{$V_1$}(dia-2-3);
 \draw[->] (dia-3-1) -- node[below]{$H_1$}(dia-3-2);

\draw[->] (dia-1-3) -- node[right]{$V_2$}(dia-2-3);
 \draw[->] (dia-1-2) -- node[left]{$L_2''$}(dia-2-2);
\draw[->] (dia-2-1) -- node[left]{$L_1''$}(dia-3-1);
\draw[->] (dia-2-2) -- node[right]{$V_1$}(dia-3-2);

\path (dia-3-1) -- node[midway]{$\ve 1''$}(dia-2-2);
\path (dia-2-2) -- node[midway]{$\ve 2''$} (dia-1-3);
\end{tikzpicture}
  \end{center}
are three pairs of composable $G$-spans. Let $M_1
\stackrel{(A_1,\Phi_1,B_1)}{\implies} M_1'$ and $M_2
\stackrel{(A_2,\Phi_2,B_2)}{\implies} M_2'$ be $G$-span \m s.
Their horizontal composition is the $G$-span \m\ 
\begin{equation}\label{eq:horizontal}
  (A_1,\Phi_1,B_1) \times_T (A_2,\Phi_2,B_2) =(A_1p_1,\Phi_1
\times_T \Phi_2, B_2p_2) \colon M_1 \times_T M_2 \implies  M_1' \times_T
M_2'
\end{equation}
graphically represented by
\begin{center}
      \begin{tikzpicture}[double distance = 1.5pt]
       \matrix (dia) [matrix of math nodes, column sep=40pt, row
       sep=25pt]{
S \xrightarrow{H_1} BG & T \xrightarrow{V_1} BG & U \xrightarrow{V_2} BG\\};
 \draw[->] (dia-1-1) to[bend left=45] node[above]{$M_1$} (dia-1-2);
  \draw[->] (dia-1-1) to[bend right=45]
  node[below]{$M_1'$}(dia-1-2);
   \node[yshift=15pt]  (A) at ($(dia-1-1)!0.5!(dia-1-2)$) {};
  \node[yshift=-15pt]  (B) at ($(dia-1-1)!0.5!(dia-1-2)$) {};
   \draw[double, -implies] (A) -- node[right]{$\Phi_1$}(B);

 \draw[->] (dia-1-2) to[bend left=45] node[above]{$M_2$} (dia-1-3);
  \draw[->] (dia-1-2) to[bend right=45]
  node[below]{$M_2'$}(dia-1-3);
   \node[yshift=15pt]  (C) at ($(dia-1-2)!0.5!(dia-1-3)$) {};
  \node[yshift=-15pt]  (D) at ($(dia-1-2)!0.5!(dia-1-3)$) {};
   \draw[double, -implies] (C) -- node[right]{$\Phi_2$}(D);
\end{tikzpicture}
\quad
  \begin{tikzpicture}[double distance = 1.5pt]
   \matrix (dia) [matrix of math nodes, column sep=40pt, row
       sep=25pt]{
S \xrightarrow{H_1} BG & U \xrightarrow{V_2} BG \\};
 \draw[->] (dia-1-1) to[bend left=45] node[above]{$M_1 \times_T M_2$} (dia-1-2);
  \draw[->] (dia-1-1) to[bend right=45]
  node[below]{$M_1' \times_T M_2'$}(dia-1-2);
   \node[yshift=15pt]  (A) at ($(dia-1-1)!0.5!(dia-1-2)$) {};
  \node[yshift=-15pt]  (B) at ($(dia-1-1)!0.5!(dia-1-2)$) {};
   \draw[double, -implies] (A) -- node[right]{$\Phi$}(B);
\end{tikzpicture}
\end{center}
where the functor $\Phi= \Phi_1 \times_T \Phi_2 \colon M_1 \times_T
M_2 \to M_1' \times_T M_2'$
is given by
\begin{equation*}
  (\Phi_1 \times_T \Phi_2)(x_1,t,x_2) = (\Phi_1x_1, A_2x_2 \circ t
  \circ B_1x_1^{-1}, \Phi_2x_2), \qquad
  x_1 \in \Ob{M_1},\, t \in T(R_1x_1,L_2x_2),\, x_2 \in \Ob{M_2}
\end{equation*}
on objects. Due to the \m s $R_1'\Phi_1x_1 \xleftarrow{B_1x_1} R_1x_1
\xrightarrow{t} L_2x_2 \xrightarrow{A_2x_2} L_2'\Phi_2x_2$ in $T$,
the right hand side is an object of $M_1' \times_T M_2'$.
 Since
\begin{multline*}
  V_2(B_2x_2) (\ve 1 \times_T \ve 2)(x_1,t,x_2) \stackrel{\text{Defn~\ref{defn:funddefn}.\eqref{funddefn:4a}}}{=}
V_2(B_2x_2) \ve 2(x_2) V_1(t) \ve 1(x_1) \\ =
\ve 2'(\Phi_2x_2) V_1(A_2x_2) V_1(t) \ve 1(x_1)  =
 \ve 2'(\Phi_2x_2) V_1(A_2x_2) V_1(t) V_1(B_1x_1^{-1}) V_1(B_1x_1)\ve
 1(x_1) \\=
 \ve 2'(\Phi_2x_2) V_1(A_2x_2) V_1(t) V_1(B_1x_1^{-1}) \ve
 1'(\Phi_1x_1) H_1(A_1x_1) =
(\ve 1' \times_T \ve 2')(\Phi_1 \times_T \Phi_2)(x_1,t,x_2) H_1(A_1x_1)
\end{multline*}
\eqref{eq:horizontal} is indeed a $G$-span \m.

If, additionally, 
$ M_1' \stackrel{(A_1',\Phi_1',B_1')}{\implies} M_1''$ and
$ M_2' \stackrel{(A_2',\Phi_2',B_2')}{\implies}M_2''$ are $G$-span \m s, then
\begin{multline*}
  ((A_1',\Phi_1',B_1') \ast(A_1,\Phi_1,B_1)) \times_T
  ((A_2',\Phi_2',B_2') \ast (A_2,\Phi_2,B_2)) \\=
((A_1',\Phi_1',B_1') \times_T (A_2',\Phi_2',B_2') ) \ast
((A_1,\Phi_1,B_1) \times_T (A_2,\Phi_2,B_2) )  \colon
M_1 \times_T M_2 \implies M_1'' \times_T M_2''
\end{multline*}
meaning that the interchange law  holds. The functor component $M_1
\times_T M_2 \to M_1'' \times_T M_2''$ of both these $G$-span \m s
take $(x_1,t,x_2) \in \Ob{M_1 \times_T M_2}$, $R_1x_1 \xrightarrow{t} L_2x_2$, to
\begin{equation*}
  (\Phi_1'\Phi_1x_1,
 R_1''\Phi_1'\Phi_1x_1 \xleftarrow{B_1'\Phi_1x_1} R_1'\Phi_1x_1
 \xleftarrow{B_1x_1} R_1x_1 \xrightarrow{t} L_2x_2
 \xrightarrow{A_2x_2} L_2'\Phi_2x_2 \xrightarrow{A_2'\Phi_2x_2}
 L_2''\Phi_2'\Phi_2x_2, \Phi_2'\Phi_2x_2)
\end{equation*}
in $M_1'' \times_T M_2''$. The two natural transformations are 
$L_1x_1 \stackrel{A_1x_1}{\implies} L_1'\Phi_1x_1
\stackrel{A_1'\Phi_1x_1}{\implies} L_1''\Phi_1'\Phi_1x_1$ and
$L_2x_2 \stackrel{B_2x_2}{\implies} L_2'\Phi_2x_2
\stackrel{B_2'\Phi_2x_2}{\implies} L_2''\Phi_2'\Phi_2x_2$ in both cases.
Thus $\mathbf{span}_G$ is a strict
$2$-category with finite groupoids over $BG$ as objects, 
$G$-spans as
$1$-\m s, and $G$-span \m s as $2$-\m s. 
\end{rmk}

\section{Examples}
\label{sec:examples}

This section contains examples of groupoid $G$-spans and their matrices.

\begin{exmp}[Realizing matrices by $G$-spans]\label{exmp:real}
  Suppose $M$ is a finite groupoid while $S$ and $T$ are finite {\em sets}. Let $L$, $R$, and
  $\ve{}$ be arbitrary functions defined on the object set of  $M$ with values in $S$, $T$,
  $G$. Then
 \begin{center}
      \begin{tikzpicture}[double distance = 1.5pt]
       \matrix (dia) [matrix of math nodes, column sep=60pt, row
       sep=25pt]{
      M & T \\
      S & BG \\};
      \draw[->] (dia-1-1) -- node[above]{$R$}(dia-1-2);
      \draw[->] (dia-1-1) -- node[left]{$L$}  (dia-2-1);
      \draw[->] (dia-1-2) -- (dia-2-2);
      \draw[->] (dia-2-1) -- (dia-2-2);
      \path (dia-2-1) -- node[midway]{$\ve{} \colon \Ob{M} \to G$} (dia-1-2);
      \end{tikzpicture}
\end{center}
is  a $G$-span with matrix
\begin{equation*}
  [M,\ve{}](c,c) = \sum_{g \in G} \chi(M\{L^{-1}c \cap \ve{}^{-1}g
  \cap R^{-1}d\})g \qquad c \in S,\, d \in T,\, g \in G
\end{equation*}
\end{exmp}

More concretely, if $K_1$, $K_2$ are finite groups and $g_1$, $g_2$ 
elements of $G$, the $(1 \times 1)$-matrix over $\Q G$ of the $G$-span
 \begin{center}
      \begin{tikzpicture}[double distance = 1.5pt]
       \matrix (dia) [matrix of math nodes, column sep=70pt, row
       sep=25pt]{
      BK_1 \amalg BK_1 \amalg BK_2 & \{1\} \\
      \{1\} & BG \\};
      \draw[->] (dia-1-1) -- node[above]{}(dia-1-2);
      \draw[->] (dia-1-1) -- node[left]{}  (dia-2-1);
      \draw[->] (dia-1-2) -- (dia-2-2);
      \draw[->] (dia-2-1) -- (dia-2-2);
      \path (dia-2-1) -- node[midway]{$\ve{}(\Ob{BK_i})=g_i \in G$} (dia-1-2);
      \end{tikzpicture}
\end{center} 
has the single entry $[BK_1 \amalg BK_1 \amalg BK_2, \ve{}](1,1)=
(\frac{2}{|K_1|}g_1+\frac{1}{|K_2|}g_2) \in \Q_{\ge 0}G$. 

Also, as the simplest instance of Example~\ref{exmp:real}, the $G$-span
 \begin{center}
      \begin{tikzpicture}[double distance = 1.5pt]
       \matrix (dia) [matrix of math nodes, column sep=70pt, row
       sep=25pt]{
      M  & \{1\} \\
      \{1\} & BG \\};
      \draw[->] (dia-1-1) -- node[above]{}(dia-1-2);
      \draw[->] (dia-1-1) -- node[left]{}  (dia-2-1);
      \draw[->] (dia-1-2) -- (dia-2-2);
      \draw[->] (dia-2-1) -- (dia-2-2);
      \path (dia-2-1) -- node[midway]{$\ve{} \colon \Ob{M} \to G$} (dia-1-2);
      \end{tikzpicture}
\end{center} 
has matrix $[M,\ve{}](1,1) = \sum_{g \in G} \chi(M \{\ve {} =g \})g$
for any function $\ve{} \colon \Ob{M} \to G$.

\begin{exmp}
  Let $H \colon S \to BG$, $V \colon T \to BG$
be  groupoid \m s and $c \in \pi_0(S)$, $d \in \pi_0(T)$ components of
$S$, $T$. The only non-zero values of the  column vector $[(S\{c\} \hookrightarrow S,e)^*] \colon
\pi_0(S) \to \Q G$ and the row vector $[(T\{d\} \hookrightarrow
T,e)_*] \colon \pi_0(T) \to \Q G$ of the $G$-spans
 \begin{center}
      \begin{tikzpicture}[double distance = 1.5pt]
       \matrix (dia) [matrix of math nodes, column sep=40pt, row
       sep=25pt]{
  S\{c\} & S\{c\} \\ S & BG \\};
  \draw[right hook ->] (dia-1-1) -- (dia-2-1);
  \draw[double] (dia-1-1) -- (dia-1-2);
  \draw[->] (dia-2-1) -- node[below]{$H$}(dia-2-2);
  \draw[->] (dia-1-2) -- node[right]{$H|S\{c\}$}(dia-2-2);
  \path (dia-2-1) -- node[midway]{$e$} (dia-1-2);
\end{tikzpicture}
\qquad \qquad
\begin{tikzpicture}[double distance = 1.5pt]
       \matrix (dia) [matrix of math nodes, column sep=40pt, row
       sep=25pt]{
  T\{d\} & T  \\ T\{d\} & BG \\};
  \draw[->] (dia-1-1) -- (dia-1-2);
  \draw[double] (dia-1-1) -- (dia-2-1);
  \draw[->] (dia-1-2) -- node[right]{$V$}(dia-2-2);
  \draw[->] (dia-2-1) -- node[below]{$V|T\{d\}$}(dia-2-2);
  \path (dia-2-1) -- node[midway]{$e$} (dia-1-2);
\end{tikzpicture}
  \end{center}
  are $[(S\{c\} \hookrightarrow S,e)^*](c)=\overline{HS(c,c)}$ and
  $[(T\{d\} \hookrightarrow T,e)_*](d) = \overline{VT(d,d)}$. The proof is similar to that of Lemma~\ref{lemma:idempotentmatrix}. 
\end{exmp}

\begin{exmp}[The universal $G$-span from $S$ to
  $T$] \label{exmp:universal}
Let $S \xrightarrow{H} BG \xleftarrow{V} T$ be a cospan of groupoids. 
  Consider the homotopy pull-back
\begin{center}
      \begin{tikzpicture}[double distance = 1.5pt]
       \matrix (dia) [matrix of math nodes, column sep=55pt, row
       sep=35pt]{
         S \times_{BG} T & T \\ S & BG \\};
       \draw[->] (dia-1-1) -- node[above]{$p_2$}(dia-1-2);
       \draw[->] (dia-1-1) -- node[left]{$p_1$}(dia-2-1);
       \draw[->] (dia-1-2) -- node[right]{$V$}(dia-2-2);
       \draw[->] (dia-2-1) -- node[below]{$H$}(dia-2-2);
       \path (dia-2-1) --node[midway]{$\ve{} \colon Hp_1 \implies Vp_2$} (dia-1-2);
   \end{tikzpicture}
\end{center}
where $\ve{}(x,k,y)=k$ for all  $(x,k,y) \in \Ob{S \times_{BG} T}=\Ob{S} \times G \times \Ob{T}$.

For $c \in \pi_0(S)$, $t \in \pi_0(T)$, the objects of the two-sided homotopy fibre $\dqt{c}{(S \times_{BG}
  T)}{d}$ are $(s,x,k,y,t)$ where $(x,k,y) \in \Ob{S \times_{BG} T}$, $s \in S(c,x)$, $t \in T(y,d)$. A \m\ $(s_1,x_1,k_1,y_1,t_1) \to (s_2,x_2,k_2,y_2,t_2)$ is a pair $(s,t) \in S(x_1,x_2) \times T(y_1,y_2)$ such that 
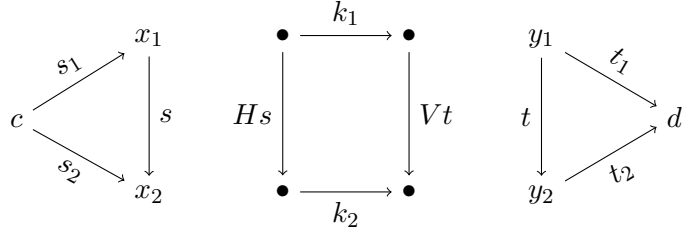
\begin{center}
      \begin{tikzpicture}[double distance = 1.5pt]
       \matrix (dia) [matrix of math nodes, column sep=35pt, row
       sep=15pt]{
     & x_1 & \bullet & \bullet & y_1 \\
     c &&&&& d \\
      & x_2 & \bullet & \bullet & y_2 \\}; 
    \draw[->] (dia-2-1) --node[above,sloped]{$s_1$} (dia-1-2);
    \draw[->] (dia-2-1) --node[below,sloped]{$s_2$} (dia-3-2);
    \draw[->] (dia-1-5) --node[above,sloped]{$t_1$} (dia-2-6);
    \draw[->] (dia-3-5) --node[below,sloped]{$t_2$} (dia-2-6);
    \draw[->] (dia-1-2) -- node[right]{$s$}(dia-3-2);
    \draw[->] (dia-1-3) -- node[left]{$Hs$} (dia-3-3);
    \draw[->] (dia-1-4) -- node[right]{$Vt$} (dia-3-4);
    \draw[->] (dia-1-5) -- node[left]{$t$} (dia-3-5);
    \draw[->](dia-1-3) -- node[above]{$k_1$} (dia-1-4);
    \draw[->](dia-3-3) -- node[below]{$k_2$} (dia-3-4);
   \end{tikzpicture}
\end{center}
commutes. There are groupoid \m s
\begin{center}
   \begin{tikzpicture}[>=stealth', baseline=(current bounding box.-2)]
  \matrix (dia) [matrix of math nodes, column sep=35pt, row sep=30pt]{
  \dqt c{(S \times_{BG} T)}d & G \\};
  \draw[->] ($(dia-1-1.east)+(0,0.1)$) -- ($(dia-1-2.west)+(0,0.08)$) 
  node[pos=.5,above] {$\dqt c{\ve{}}d$}; 
  \draw[->] ($(dia-1-2.west)-(0,0.15)$) -- ($(dia-1-1.east)-(0,0.13)$)
   node[pos=.5,below] {} ; 
 \end{tikzpicture}
\end{center}
where the functor with domain $G$ takes $k \in G$ to $(c \xrightarrow{1} c, c,k,d, d \xrightarrow{1} d)$.
The composition from $G$ to $G$ is manifestly the identity functor. 
The commutative diagrams
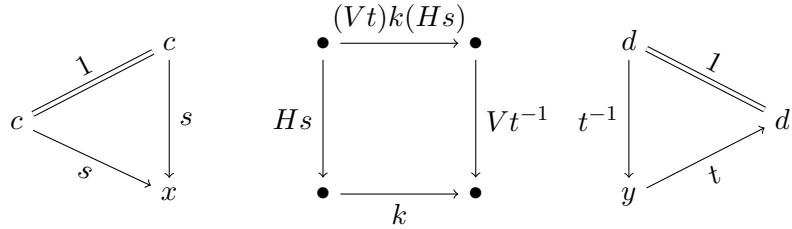
\begin{center}
      \begin{tikzpicture}[double distance = 1.5pt]
       \matrix (dia) [matrix of math nodes, column sep=45pt, row
       sep=15pt]{
     & c & \bullet & \bullet & d \\
     c &&&&& d \\
      & x & \bullet & \bullet & y  \\}; 
    \draw[double] (dia-2-1) --node[above,sloped]{$1$} (dia-1-2);
    \draw[->] (dia-2-1) --node[below,sloped]{$s$} (dia-3-2);
    \draw[double] (dia-1-5) --node[above,sloped]{$1$} (dia-2-6);
    \draw[->] (dia-3-5) --node[below,sloped]{$t$} (dia-2-6);
    \draw[->] (dia-1-2) -- node[right]{$s$}(dia-3-2);
    \draw[->] (dia-1-3) -- node[left]{$Hs$} (dia-3-3);
    \draw[->] (dia-1-4) -- node[right]{$Vt^{-1}$} (dia-3-4);
    \draw[->] (dia-1-5) -- node[left]{$t^{-1}$} (dia-3-5);
    \draw[->](dia-1-3) -- node[above]{$(Vt)k(Hs)$} (dia-1-4);
    \draw[->](dia-3-3) -- node[below]{$k$} (dia-3-4);
   \end{tikzpicture}
\end{center}
define a natural transformation between the identity functor and the other composition.  These groupoid equivalences restrict to equivalences between the full subgroupoids $(\dqt{c}{(S \times_{BG}
  T)}{d})\{\dqt{c}{\ve{}}{d}=g\}$ and $\{g\}$
for any $g \in G$. We conclude that
\begin{equation*}
  [S \times_{BG} T,\ve{}] (c,d)= \frac{1}{|T(d,d)|} \sum G 
  \stackrel{\text{\eqref{eq:defnUbar}}}{=} \frac{|G|}{|T(d,d)|} \overline{G}, \qquad c \in
  \pi_0(S),\, d \in \pi_0(T),
\end{equation*}
\end{exmp}

The next example uses {\em action groupoids}. 

\begin{defn}\label{defn:actiongroupoidII}
 The action groupoid $X//G$ of a right $G$-set $X$ is the groupoid
 with object set $X$ and \m s $(X//G)(x_1,x_2) = \{g \in G \mid
 x_2g=x_1\}$ for all $x_1,x_2 \in X$.
\end{defn}
 
 Viewing $X$ as a set-valued
 functor on $BG$ and $X//G = \int_{BG} X$ as its Grothendieck
 construction,
 the \Euc\ is $\chi(X//G) = \chi( \int_{BG}X) = |X|/|G|$ by \cite[Proposition~2.8]{leinster08}.

In particular, for a subgroup $H$ of $G$, $\act HG$ is the action
groupoid for the right $G$-set $\qt HG = \{Hg \mid g \in G\}$ of right
cosets of $H$ in $G$. In this case, the \m\ sets are double $H$-cosets
\begin{equation*}
   (\act
    HG)(Hg_1,Hg_2)=g_2^{-1}Hg_1, \quad g_1,g_2 \in G
  \end{equation*}
and the \Euc\ is $\chi(\act HG)=\frac{1}{|H|}$.
The extreme cases are $EG=\act {\{e\}}G = G//G$, whose objects are the elements of $G$, and $BG=\act GG$ with $G/G$ as its only object. When
$H \leq K \leq G$,
there is a functor $\act HG \to \act KG$,
surjective on object sets, $H \backslash G \to K \backslash G$, and
injective on \m\ sets, $g_2^{-1}Hg_1 \subseteq g_2^{-1}Kg_1$

\begin{exmp}[$(1 \times 1)$-matrices]\label{exmp:-1matrix}
  If $S \xrightarrow{H} G \xleftarrow{V} T$ is a cospan of
  groups, there is a  group action $ G \times (S \times T) \to G$ given by
  $ x \cdot (s,t)  = V(t)^{-1}xH(s)$ for all $x \in G$, $s \in S$,
  $t \in T$. Suppose $M$ is a finite $(S \times T)$-invariant subset of $G$
  and let $M//(S \times T)$ be the action groupoid with
  $\Ob{M//(S \times T)}=M$ and  \m s 
  $M//(S \times T)(x_1,x_2) = \{(s,t) \mid   x_2 \cdot (s,t) = x_1 \} =
\{(s,t) \mid
  V(t)x_1=x_2H(s)\}$. There is a $G$-span
 \begin{center}
      \begin{tikzpicture}[double distance = 1.5pt]
       \matrix (dia) [matrix of math nodes, column sep=45pt, row
       sep=25pt]{
      M//(S \times T) &\bullet//T \\
      \bullet//S & \bullet//G \\};
      \draw[->] (dia-1-1) -- node[above]{$R$} (dia-1-2);
      \draw[->] (dia-1-1) -- node[left]{$L$}(dia-2-1);
      \draw[->] (dia-1-2) -- node[right]{$V$}(dia-2-2);
      \draw[->] (dia-2-1) -- node[below]{$H$}(dia-2-2);
      \path (dia-2-1) -- node[midway]{$\ve{}(x)=x$} (dia-1-2);
      \end{tikzpicture}
\end{center}
where $L(s,t)=s$, $R(s,t)=t$ for all \m s $(s,t)$ in $M//(S \times
T)$. The inclusion $\ve{} \colon M \hookrightarrow G$, is a natural transformation of functors $HL \implies VR$ as
\begin{center}
      \begin{tikzpicture}[double distance = 1.5pt]
       \matrix (dia) [matrix of math nodes, column sep=45pt, row
       sep=25pt]{
      \bullet & \bullet \\
      \bullet & \bullet \\};
      \draw[->] (dia-1-1) -- node[above]{$x_1$} (dia-1-2);
      \draw[->] (dia-2-1) -- node[below]{$x_2$} (dia-2-2);
      \draw[->] (dia-1-1) -- node[left]{$H(s)=HL(s,t)$} (dia-2-1);
       \draw[->] (dia-1-2) -- node[right]{$VR(s,t)=V(t)$} (dia-2-2);
      \end{tikzpicture}
\end{center}
commutes in $BG=\bullet//G$ for every $(s,t) \in M//(S \times T)(x_1,x_2)$.

The object set of the two-sided homotopy fibre $\dqt{\bullet}{(M//(S
  \times T)}{\bullet}$ is $S \times M \times T$ and the \m\ sets
\begin{equation*}
  (M//(S \times T))((s_1,x_1,t_1),(s_2,x_2,t_2)) =\{(s,t) \mid
V(t)x_1=x_2H(s), ss_1=s_2, t_2t=t_1\}
\end{equation*}
contain at most one \m . Thus $\dqt{\bullet}{(M//(S
  \times T)}{\bullet} = \pi_0(\dqt{\bullet}{(M//(S
  \times T)}{\bullet})$.
The function
$(\dqt{\bullet}{\ve{}}{\bullet})(s,x,t)= x \cdot (s,t^{-1}) =
V(t)xH(s) \colon S\times M \times T \to M$ induces a bijective map
$ \pi_0( \dqt{\bullet}{(M//(S \times T)}{\bullet}) \to M $ as in Lemma~\ref{lemma:twosidedfibre}. Indeed,
the functor 
$\dqt{\bullet}{\ve{}}{\bullet} \colon \dqt{\bullet}{(M//(S \times
  T)}{\bullet} \to M$ and the functor
$ M \to \dqt{\bullet}{(M//(S \times T)}{\bullet}$ taking $x$ to
$(e,x,e)$ are connected by a natural iso\m\
\begin{center}
      \begin{tikzpicture}[double distance = 1.5pt]
       \matrix (dia) [matrix of math nodes, column sep=25pt, row
       sep=15pt]{
        & \bullet &(\dqt{\bullet}{\ve{}}{\bullet}) (s,x,t) & \bullet \\
        \bullet &&&& \bullet \\
        & \bullet & x & \bullet \\};
       \draw[->] (dia-2-1) -- node[sloped,above]{$e$} (dia-1-2);
       \draw[->] (dia-1-4) -- node[sloped,above]{$e$} (dia-2-5);
       \draw[->] (dia-2-1) -- node[sloped,below]{$s$} (dia-3-2);
       \draw[->] (dia-3-4) -- node[sloped,below]{$t$} (dia-2-5);
       \draw[->] (dia-1-2) -- node[right]{$s$} (dia-3-2);
       \draw[->] (dia-1-4) -- node[left]{$t^{-1}$} (dia-3-4);
       \draw[->] (dia-1-3) -- node[left]{$(s,t^{-1})$} (dia-3-3);
      \end{tikzpicture}
\end{center}
This equivalence of groupoids restricts to an equivalence between $(\dqt{\bullet}{(M//(S
  \times T)}{\bullet})\{ \dqt{\bullet}{\ve{}}{\bullet}=x\}$ and
$\{x\}$ for all $x \in M$. If $g \in G$ lies outside $M$ then $(\dqt{\bullet}{(M//(S
  \times T)}{\bullet})\{ \dqt{\bullet}{\ve{}}{\bullet}=x\}$ is the
empty groupoid with \Euc\ $0$.

We have now seen that the entry of the $(1 \times 1)$-matrix associated to the
$G$-span of $M//(S \times T)$ is
\begin{equation*}
  [M//(S \times T), \ve{}](\bullet,\bullet) = \frac{1}{|T|}  \sum M 
\end{equation*}
In particular, if $S=T=\{e\}$ is the trivial subgroup of $G$ and $M$
any finite set of
elements of $G$, the matrix of the $G$-span
\begin{center}
      \begin{tikzpicture}[double distance = 1.5pt]
       \matrix (dia) [matrix of math nodes, column sep=45pt, row
       sep=25pt]{
      M & \{1\}\\
      \{1\} & BG \\};
      \draw[->] (dia-1-1) -- (dia-1-2.175);
      \draw[->] (dia-1-1) -- (dia-2-1);
      \draw[->] (dia-1-2) -- (dia-2-2);
      \draw[->] (dia-2-1) -- (dia-2-2);
      \path (dia-2-1) -- node[midway]{$\ve{}(x)=x$} (dia-1-2);
      \end{tikzpicture}
\end{center}
is the $(1 \times 1)$-matrix $[M,\ve{}](1,1)=\sum M$. For instance,
for any   $x \in G$, the $G$-span 
\begin{center}
      \begin{tikzpicture}[double distance = 1.5pt]
       \matrix (dia) [matrix of math nodes, column sep=45pt, row
       sep=25pt]{
      \{x\} & \{1\} \\
      \{1\} & BG \\};
      \draw[->] (dia-1-1) -- (dia-1-2);
      \draw[->] (dia-1-1) -- (dia-2-1);
      \draw[->] (dia-1-2) -- (dia-2-2);
      \draw[->] (dia-2-1) -- (dia-2-2);
      \path (dia-2-1) -- node[midway]{$\ve{}(x)=x$} (dia-1-2);
      \end{tikzpicture}
\end{center}
realizes
the $(1 \times 1)$-matrix $(x)$ over $\Q G$. (Taking $G=\sqrt[4]{1}$
to be the complex $4$th roots of unity, we recover an example from Section~\ref{sec:introduction}.)
\end{exmp}

\begin{exmp}[Spans of groups] 
Consider the two diagrams
  \begin{center}
      \begin{tikzpicture}[double distance = 1.5pt]
       \matrix (dia) [matrix of math nodes, column sep=65pt, row
       sep=25pt]{
      M & K_2 && BM & BK_2 \\ K_1 & G && BK_1 & BG \\};
     \draw[->] (dia-1-1) --node[above]{$R$} (dia-1-2);
     \draw[->] (dia-1-4) --node[above]{$BR$} (dia-1-5);
     \draw[->](dia-2-1) -- node[below]{$H$}(dia-2-2);
     \draw[->](dia-2-4) -- node[below]{$BH$}(dia-2-5);
     \draw[->](dia-1-1) --node[left]{$L$} (dia-2-1);
     \draw[->](dia-1-4) --node[left]{$BL$} (dia-2-4);
     \draw[->](dia-1-2) -- node[right]{$V$}(dia-2-2);
     \draw[->](dia-1-5) -- node[right]{$BV$}(dia-2-5);
     \path (dia-2-4) --node[midway]{$\ve{}=x$} (dia-1-5);
     \path (dia-2-1) -- node[midway]{$HL=(VR)^x$} (dia-1-2);
\end{tikzpicture}
  \end{center}
On the left is a diagram of finite groups that is commutative up to
conjugation in the sense
that there is an $x \in G$ such that
$x(HL)(m)=(VR)(m)x$ for all $m \in
M$. The diagram on the right, induced from the one on the right, is a
$G$-span where $\ve{}$ is the constant function with value $x$. 

The object set of $\dqt{\bullet}{BM}{\bullet}$ is $K_1 \times K_2$ and
$(\dqt{\bullet}{\ve{}}{\bullet})(k_1,k_2)=V(k_2)xH(k_1)$. 
Also, the group $M$ acts on the set $K_1 \times K_2$, $m \cdot (k_1,k_2) =
(L(m)k_1,k_2R(m)^{-1})$, and
\begin{equation*}
  \dqt{\bullet}{BM}{\bullet} = (K_1 \times K_2)//M \qquad
  (\dqt{\bullet}{BM}{\bullet})\{\dqt{\bullet}{\ve{}}{\bullet}=g\} =
  \{(k_1,k_2) \in K_1 \times K_2 \mid V(k_2)xH(k_1)= g\}//M
\end{equation*}
are action groupoids. We conclude that
\begin{equation*}
  [BM,\ve{}](\bullet,\bullet) = \frac{1}{|M||K_2|} \sum_{g \in G} |\{(k_1,k_2) \in K_1 \times K_2 \mid V(k_2)xH(k_1)=g\} |g
\end{equation*}
according to Definition~\ref{defn:funddefn}.\eqref{funddefn:6}.

Now suppose we have composable spans of groups 
\begin{center}
      \begin{tikzpicture}[double distance = 1.5pt]
       \matrix (dia) [matrix of math nodes, column sep=65pt, row
       sep=30pt]{
     & M_2 & K_3 \\
     M_1 & K_2 & G \\
     K_1 & G \\};
    \draw[->] (dia-1-2) -- node[above]{$R_2$}(dia-1-3);
    \draw[->] (dia-2-1) -- node[above]{$R_1$}(dia-2-2);
    \draw[->] (dia-2-2) -- node[below]{$V$} (dia-2-3);
    \draw[->] (dia-3-1) -- node[below]{$H_1$}(dia-3-2);
    \draw[->] (dia-2-1) -- node[left]{$L_1$}(dia-3-1);
    \draw[->] (dia-2-2) --node[right]{$V$} (dia-3-2);
    \draw[->] (dia-1-2) -- node[left]{$L_2$}(dia-2-2);
    \draw[->] (dia-1-3) -- node[right]{$V_2$}(dia-2-3);
    \path (dia-3-1) -- node[midway]{$H_1L_1=(VR_1)^{x_1}$}(dia-2-2);
    \path (dia-2-2) -- node[midway]{$VL_2=(V_2R_2)^{x_2}$}(dia-1-3);
    \end{tikzpicture}
\end{center}
where we can find $x_1,x_2\in  G$ such that $x_1(H_1L_1)(m_1)=(VR_1)(m_1)x_1$ for all $m_1 \in M_1$, and
$x_2(VL_2)(m_2)=(V_2R_2)(m_2)x_2$ for all $m_2 \in M_2$. For all
$(m_1,m_2) \in M_1 \times_{K_2} M_2$, $R_1m_1=L_2m_2$, and, as
$x_2x_1(H_1L_1)(m_1)=(V_2R_2)(m_2)x_2x_1$,
we may extract the two diagrams 
 \begin{center}
      \begin{tikzpicture}[double distance = 1.5pt]
       \matrix (dia) [matrix of math nodes, column sep=85pt, row
       sep=30pt]{
      M_1 \times_{K_2} M_2 & K_3 & B(M_1 \times_{K_2} M_2)  & BK_3 \\ K_1 & G & BK_1 & BG \\};
     \draw[->] (dia-1-1) --node[above]{$R_2p_2$} (dia-1-2);
     \draw[->] (dia-1-3) --node[above]{$B(R_2p_2)$} (dia-1-4);
     \draw[->](dia-2-1) -- node[below]{$H_1$}(dia-2-2);
     \draw[->](dia-2-3) -- node[below]{$BH_1$}(dia-2-4);
     \draw[->](dia-1-1) --node[left]{$L_1p_1$} (dia-2-1);
     \draw[->](dia-1-3) --node[left]{$B(L_1p_1)$} (dia-2-3);
     \draw[->](dia-1-2) -- node[right]{$V_2$}(dia-2-2);
     \draw[->](dia-1-4) -- node[right]{$BV_2$}(dia-2-4);
     \path (dia-2-4) --node[midway]{$\ve{}=x_2x_1$} (dia-1-3);
     \path (dia-2-1) -- node[midway]{$H_1L_1p_1=(V_2R_2p_2)^{x_2x_1}$} (dia-1-2);
\end{tikzpicture}
  \end{center}
Since $[B(M_1 \times_{K_2} M_2), \ve{}=x_2x_1]=[BM_1,\ve{}=x_1][BM_2,\ve{}=x_2]$ by Theorem~\ref{thm:main}
we arrive at the combinatorial identities
\begin{multline*}
  |M_1||K_2||M_2| |\{(k_1,k_3) \in K_1 \times K_3 \mid
    V_2(k_3)x_2x_1H_1(k_1)=g\}| \\ =
  |M_1 \times_{K_2} M_2|
  \sum_{\substack{g_1,g_2 \in G \\ g_1g_2=g}}
   |\{(k_1,k_2) \in K_1 \times K_2 \mid V(k_2)x_1H_1(k_1)=g_1\}|
     |\{(k_2,k_3) \in K_2 \times K_3 \mid V_2(k_3)x_2V(k_2)=g_2\}|
\end{multline*}
valid for all $g \in G$.
\end{exmp}

\begin{exmp}[Spans of action groupoids]
  Let $H_1$, $K_1$ and $K_2$ be subgroups of $G$ with $H_1 \leq K_1 \cap
  K_2$ so that 
  \begin{center}
       \begin{tikzpicture}[double distance = 1.5pt]
       \matrix (dia) [matrix of math nodes, column sep=50pt, row
       sep=25pt]{
         \act  {H_1}G & \act {K_2}G\\ \act {K_1}G & \act GG \\};
       \draw[->] (dia-1-1) --node[above]{$R$}(dia-1-2);
       \draw[->] (dia-1-1) -- node[left]{$L$}(dia-2-1);
       \draw[->] (dia-1-2) -- node[right]{$V$}(dia-2-2);
       \draw[->] (dia-2-1) -- node[below]{$H$}(dia-2-2);
        \path (dia-2-1) -- node[midway]{$\ve 1 \colon H_1
          \backslash G \to \{e\}$} (dia-1-2);
   \end{tikzpicture}
  \end{center}
is a $G$-span of action groupoids.  The function $\ve{1}$ is constant with value the neutral
element $e \in G$ since the diagram commutes.

The objects of the two-sided homotopy fibre 
$\dqt{K_1g_1}{(\act {H_1}G)}{K_2g_2}$ are all triples $(x^{-1}k_1g_1,H_1x,g_2^{-1}k_2x)$
where $k_1 \in K_1$, $k_2 \in K_2$, $x \in G$, and 
\begin{equation*}
(\dqt{K_1g_1}{\ve{1}}{K_2g_2})
(x^{-1}k_1g_1,H_1x,g_2^{-1}k_2x)=g_2^{-1}k_2k_1g_1 = g_2^{-1}
(\dqt{K_1}{\ve{1}}{K_2})(x^{-1}k_1,k_2x)g_1  
\end{equation*}
by Definition~\ref{defn:funddefn}.\eqref{funddefn:2}.
For all $g_1,g_2,g \in G$ there are obvious iso\m s of groupoids
\begin{align*}
  \dqt{K_1}{(\act {H_1}G)}{K_2} & \longleftrightarrow
  \dqt{K_1g_1}{(\act {H_1}G)}{K_2g_2} \\
  (\dqt{K_1}{(\act {H_1}G)}{K_2})\{ \dqt{K_1}{\ve{1}}{K_2}=g_2gg_1^{-1}\} &
                                                            \longleftrightarrow 
  (\dqt{K_1g_1}{(\act {H_1}G)}{K_2g_2})\{ \dqt{K_1g_1}{\ve{1}}{K_2g_2}=g\}
\end{align*}
Thus we only need to consider the two-sided homotopy fibre
$\dqt{K_1}{(\act{H_1}{G})}{K_2}$. We note first that there is at most
one \m\ in any \m\ set so that
$\dqt{K_1}{(\act{H_1}{G})}{K_2}$ is equivalent to its component set. 
The commutative
diagram in $\dqt{K_1}{(\act {H_1}G)}{K_2}$
\begin{center}
     \begin{tikzpicture}[double distance = 1.5pt]
       \matrix (dia) [matrix of math nodes, column sep=30pt, row
       sep=15pt]{
     & K_1x & H_1x & K_2x \\
     K_1 &&&& K_2 \\
     & K_1 & H_1 & K_2 \\};
     \draw[->] (dia-2-1) -- node[above,sloped]{$x^{-1}k_1$}
     (dia-1-2);
     \draw[->] (dia-2-1) -- node[below,sloped]{$k_1$} (dia-3-2);
    \draw[->] (dia-1-4) -- node[above,sloped]{$k_2x$}
     (dia-2-5);
     \draw[->] (dia-3-4) -- node[below,sloped]{$k_2$}
     (dia-2-5);
     \draw[->] (dia-1-2) -- node[right]{$x$}(dia-3-2); 
     \draw[->] (dia-1-3) -- node[right]{$x$}(dia-3-3); 
     \draw[->] (dia-1-4) -- node[left]{$x$}(dia-3-4); 
\end{tikzpicture}
\end{center}
shows that  $(x^{-1}k_1,H_1x,k_2x) \cong
(k_1,H_1,k_2)$ and  another  diagram
\begin{center}
     \begin{tikzpicture}[double distance = 1.5pt]
       \matrix (dia) [matrix of math nodes, column sep=30pt, row
       sep=15pt]{
     & K_1 & H_1 & K_2 \\
     K_1 &&&& K_2 \\
     & K_1 & H_1 & K_2 \\};
     \draw[->] (dia-2-1) -- node[above,sloped]{$k_{11}$}
     (dia-1-2);
     \draw[->] (dia-2-1) -- node[below,sloped]{$k_{12}$} (dia-3-2);
    \draw[->] (dia-1-4) -- node[above,sloped]{$k_{21}$}
     (dia-2-5);
     \draw[->] (dia-3-4) -- node[below,sloped]{$k_{22}$}
     (dia-2-5);
     \draw[->] (dia-1-2) -- node[right]{$h_1$}(dia-3-2); 
     \draw[->] (dia-1-3) -- node[right]{$h_1$}(dia-3-3); 
     \draw[->] (dia-1-4) -- node[left]{$h_1$}(dia-3-4); 
\end{tikzpicture}
\end{center}
shows that
\begin{equation*}
  (k_{11},H_1,k_{21}) \cong (k_{12},H_1,k_{22}) \iff
  \exists h_1 \in H_1 \colon (h_1k_{11},k_{21}h_1^{-1}) = (k_{12},k_{22}) 
\end{equation*}
The conclusion is that $(k_1,k_2) \to (k_1,H_1,k_2)$ induces
a bijection
\begin{equation*}
  \qt{H_1}{(K_1 \times K_2)} \stackrel{\cong}{\longrightarrow} \pi_0(\dqt{K_1}{(\act {H_1}G)}{K_2})
\end{equation*}
between
 the set of orbits for the free left action $h_1 \cdot (k_1,k_2) =
(h_1k_1,k_2h_1^{-1})$ of $H_1$ on $K_1 \times K_2$ and
 the component set
of the two-sided homotopy fibre $\dqt{K_1}{(\act {H_1}G)}{K_2}$.
The \Euc\ is
\begin{equation*}
  \chi(\dqt{K_1}{(\act {H_1}G)}{K_2}) = |\pi_0(\dqt{K_1}{(\act {H_1}G)}{K_2})|
  = |\qt{H}{(K_1 \times K_2)}| = \frac{|K_1||K_2|}{|H_1|}
\end{equation*}
The function of Lemma~\ref{lemma:twosidedfibre},
$
  \pi_0(\dqt{K_1}{\ve{1}}{K_2}) \colon \pi_0(\dqt{K_1}{(\act {H_1}G)}{K_2}) =
  \qt{H_1}{(K_1 \times K_2)} \to G
$,
takes $H_1(k_1,k_2)$ to $k_2k_1$ and induces a bijection
\begin{equation*}
  \qt{H}{\{ (k_1,k_2) \in K_1 \times K_2 \mid k_2k_1=g \}}
  \stackrel{\cong}{\to}
  \pi_0(\dqt{K_1}{(\act {H_1}G}){K_2} )\{ \dqt{K_1}{\ve{1}}{K_2}=g\}
\end{equation*}
so the \Euc\ is 
\begin{equation*}
  \chi(\dqt{K_1}{(\act {H_1}G)}{K_2})\{\dqt{K_1}{\ve{1}}{K_2}=g\}) = 
 \frac{|\{ (k_1,k_2) \in K_1 \times K_2 \mid k_2k_1=g \}|}{|H_1|} 
\end{equation*}
We can now say that 
\begin{equation*}
  [\act{H_1}G,\ve{1}](K_1g_1,K_2g_2) = 
\frac{1}{|H_1||K_2|} \sum_{g \in G}
   |\{ (k_1,k_2) \in K_1 \times K_2 \mid k_2k_1 =  g_2gg_1^{-1}\}|g
\end{equation*}


Suppose further $G$ has subgroups, $H_2$ and $K_3$, with $H_2 \leq K_2 \cap
K_3$. The composable $G$-spans 
  \begin{center}
       \begin{tikzpicture}[double distance = 1.5pt]
       \matrix (dia) [matrix of math nodes, column sep=50pt, row
       sep=25pt]{
       & \act{H_2}G & \act{K_3}G \\
         \act  {H_1}G & \act {K_2}G & \act GG\\ 
        \act {K_1}G & \act GG \\};
       \draw[->] (dia-2-1) --node[above]{$R$}(dia-2-2);
       \draw[->] (dia-2-1) -- node[left]{$L$}(dia-3-1);
       \draw[->] (dia-2-2) -- node[right]{$V$}(dia-3-2);
       \draw[->] (dia-3-1) -- node[below]{$H$}(dia-3-2);
        \draw[->] (dia-1-2) -- (dia-1-3);
        \draw[->] (dia-1-2) -- (dia-2-2);
        \draw[->] (dia-1-3) -- (dia-2-3);
        \draw[->] (dia-2-2) -- node[below]{$V$}(dia-2-3);
        \path (dia-3-1) -- node[midway]{$\ve 1 \colon H_1
          \backslash G \to \{e\}$} (dia-2-2);
         \path (dia-2-2) -- node[midway]{$\ve 2 \colon H_2
          \backslash G \to \{e\}$} (dia-1-3);
   \end{tikzpicture}
  \end{center}
produce a $G$-span $(\act{H_1}G)
\times_{\act{K_2}G} (\act{H_2}G)$ from $\act {K_1}G$ to $\act{K_3}G$. 
The group $H_1 \times H_2$ acts on $K_1 \times K_2
\times K_3$ by the rule $(h_1,h_2) \cdot (k_1,k_2,k_3)=(h_1k_1,h_2k_2h_1^{-1},k_3h_2^{-1})$.
In much the same way as before we find
that any object of the two sided homotopy fibre $\dqt{K_1}{(\act{H_1}G
  \times_{\act{K_2}G} \act{H_2}G)}{K_3}$ is isomorphic to  one of the
form $(k_1,H_1,k_2,H_2,k_3)$ 
\begin{center}
      \begin{tikzpicture}[double distance = 1.5pt]
       \matrix (dia) [matrix of math nodes, column sep=30pt, row
       sep=15pt]{
     & K_1x_1 & H_1x_1 & K_2x_1 & K_2x_2 & H_2x_2 & K_3x_2  \\
     K_1 &&&&&&& K_3 \\
      & K_1 & H_1 & K_2 & K_2 & H_2 & K_3  \\}; 
    \draw[->] (dia-2-1) --node[above,sloped]{$x_1^{-1}k_1$} (dia-1-2);
    \draw[->] (dia-2-1) --node[below,sloped]{$k_1$} (dia-3-2);
    \draw[->] (dia-1-7) --node[above,sloped]{$k_3x_2$} (dia-2-8);
    \draw[->] (dia-3-7) --node[below,sloped]{$k_3$} (dia-2-8);
    \draw[->] (dia-1-2) -- node[right]{$x_1$}(dia-3-2);
    \draw[->] (dia-1-3) -- node[right]{$x_1$}(dia-3-3);
    \draw[->] (dia-1-4) -- node[right]{$x_1$}(dia-3-4);
    \draw[->] (dia-1-5) -- node[right]{$x_2$}(dia-3-5);
    \draw[->] (dia-1-6) -- node[right]{$x_2$}(dia-3-6);
    \draw[->] (dia-1-7) -- node[left]{$x_2$}(dia-3-7);
   \draw[->] (dia-1-4) -- node[above]{$x_2^{-1}k_2x_2$}(dia-1-5);
   \draw[->] (dia-3-4) -- node[below]{$k_2$}(dia-3-5);
   \end{tikzpicture}
\end{center}
where $k_1 \in (\act{K_1}G)(K_1,K_1)=K_1$,
$k_2 \in (\act{K_2}G)(K_2,K_2)=K_2$, $k_3 \in (\act{K_3}G)(K_3,K_3)=K_3$,
and that two such objects are isomorphic if and only the corresponding
$(k_1,k_2,k_3)$ lie in the same $(H_1 \times H_2)$-orbit.
These observations can be used to determine the
 function from Lemma~\ref{lemma:twosidedfibre} 
\begin{gather*}
  \pi_0( \dqt{K_1}{(\ve 1 \times_{\act{K_2}G} \ve 2)}{K_3}) \colon 
  \pi_0( \dqt{K_1}{(\act{H_1}G \times_{\act{K_2}G} \act{H_2}G)}{K_3}) = 
\qt{(H_1 \times
    H_2)}{(K_1 \times K_2 \times K_3)} \to G  \\
   (H_1 \times H_2)(k_1,k_2,k_3) \to k_3k_2k_1 
\end{gather*}
and the \Euc\
\begin{multline*}
  \chi( (\dqt{K_1}{(\act{H_1}G)
\times_{\act{K_2}G} (\act{H_2}G)}{K_3}) \{  \dqt{K_1}{(\ve 1
  \times_{\act{K_2}G} \ve 2)}{K_3} =g\} ) \\= \frac{
|\{(k_1,k_2,k_3) \in K_1 \times K_2 \times K_3
  \mid k_3k_2k_1=g \}| }{|H_1 || H_2|}
\end{multline*}
 for any $g \in G$.

The identity
\begin{align*}
   \frac{1}{|K_3|} & \frac{1}{|H_1|}
  \frac{1}{|H_2|}
  \sum_{g \in G}  |\{(k_1,k_2,k_3) \in K_1 \times K_2 \times K_3
  \mid k_3k_2k_1=g \}|g  \\ & \stackrel{\text{Thm~\ref{thm:main}}}{=}
  [\act{H_1}G \times_{\act{K_2}G} \act{H_2}G, \ve 1
  \times_{\act{K_2}G} \ve 2] (K_1,K_3)\\
 & =[\act{H_1}G,\ve 1](K_1,K_2)[\act{H_2}G,\ve 2](K_2,K_3) \\ &=
  \frac{1}{|K_2|}\frac{1}{|H_1}\frac{1}{|K_3}\frac{1}{|H_2|}
  \sum_{g \in G} \big(\sum_{g_2g_1=g} |\{(k_1,k_2) \mid k_2k_1=g_1 \}|
  |\{(k_2,k_3) \mid k_3k_2=g_2\}| \big) g
\end{align*}
which is equivalent to
\begin{equation*}
  |K_2||\{(k_1,k_2,k_3)  \mid k_3k_2k_1 = g\}| =
\sum_{\substack{g_1,g_2 \in G \\ g_1g_2=g}}
|\{(k_1,k_2) \mid k_2k_1=g_1 \}| |\{(k_2,k_3) \mid k_3k_2=g_2 \}| 
\end{equation*}
can also, of course, be verified directly from the $|K_2|$-to-$1$ map
\begin{equation*}
 \coprod_{\substack{g_1,g_2 \in G \\
    g_2g_1=g}}
(\{(k_1,k_{21})  \mid k_{21}k_1=g_1\} \times
  \{(k_{22},k_3)  \mid
k_3k_{22}=g_2\}) \to \{(k_1,k_2,k_3) \mid k_3k_2k_1 = g\}  \\
\end{equation*}
taking $(k_1,k_{21},k_{22},k_3)$ to  $(k_1, k_{22} k_{21},k_3)$. 
\end{exmp}

Next is an example of a $O(1)$-span, where
$O(1)=\{ \pm 1\}$, with invertible complex matrix.  The following groupoids will be needed:

\begin{description}

\item[$\mathbf{Fin}$] The groupoid of finite sets
  and bijections.

\item[$\mathbf{Fin}_n$] The full subgroupoid of $\mathbf{Fin}$ on all
  sets of cardinality $n$.

\item[$\mathbf{FinPerm}$] An object of the groupoid $\mathbf{FinPerm}$
  is a pair  $(X,\sigma)$ where $\sigma$ is a
  permutation of the finite set $X$. A \m\
  $(X_1,\sigma_1) \to (X_2,\sigma_2)$  is a bijection \func
  f{X_1}{X_2} such that $\sigma_1 = \sigma_2^f$.

\item[$\mathbf{FinPerm}_{n,k}$] The full subgroupoid of
  $\mathbf{FinPerm}$ of all $(X,\sigma)$ where $X$ has
  cardinality $n$ and $\sigma$ has exactly $k$ cycles.

\item[$\mathbf{FinRel}$] An object of the groupoid $\mathbf{FinRel}$
  is a pair $(X,\rho)$ where 
 $\rho \subseteq X \times X$ is an equivalence relation on the finite set $X$. A \m\
 $(X_1,\rho_1) \to (X_2,\rho_2)$ is a bijection \func f{X_1}{X_2} such that
 $(f \times f)\rho_1=\rho_2$.

\item[$\mathbf{FinRel}_{n,k}$] The full subgroupoid of
  $\mathbf{FinRel}$ of all $(X,\rho)$ where $X$ has
  cardinality $n$ and $\rho$ has exactly $k$ classes.

\end{description}

For any finite set $X$,
the
group $\Sigma(X)$ of permutations of $X$ acts on
\begin{itemize}
\item $S_1(X,k)$, the set of permutations of $X$
with $k$ cycles
\item $S_2(X,k)$, the set of equivalence relations on $X$
with $k$ classes 
\end{itemize}

\begin{prop}\label{prop:FinPR}
  Fix a set $X$ of cardinality $n$.
  \begin{itemize}
  \item  $\mathbf{FinPerm}_{n,k}$ is equivalent to the action groupoid
  $S_1(X,k)//\Sigma(X)$. 
\item $\mathbf{FinRel}_{n.k}$
 is equivalent to the action groupoid $S_2(X,k)//\Sigma(X)$.
  \end{itemize}
\end{prop}
\begin{proof}
Any object $(Y,\tau)$ of $\mathbf{FinPerm}_{n,k}$ is isomorphic to one
of the form $(X,\sigma)$ for some permutation $\sigma \in S_1(X,k)$
  \begin{center}
      \begin{tikzpicture}[double distance = 1.5pt]
       \matrix (dia) [matrix of math nodes, column sep=45pt, row
       sep=25pt]{
       X & Y \\ X  &  Y\\};
       \draw[->] (dia-1-1) -- node[above]{$f$} (dia-1-2);
       \draw[->] (dia-2-1) -- node[below]{$f$} (dia-2-2);
       \draw[->] (dia-1-1) -- node[left]{$\tau^f$}
       (dia-2-1);
       \draw[->] (dia-1-2) -- node[right]{$\tau$} (dia-2-2);
   \end{tikzpicture}
\end{center}
and $(S_1(X,k)//\Sigma(X))(\sigma_1,\sigma_2) =
\mathbf{FinPerm}((X,\sigma_1),(X,\sigma_2))$ for any two elements 
$\sigma_1,\sigma_2 \in S_1(X,k)$.

Any object $(Y,T)$ of $\mathbf{FinRel}_{n,k}$ is isomorphic to one of
the form $(X,\rho)$ for some relation $\rho \in S_2(X,k)$ and 
$(S_2(X,k)//\Sigma(X))(\rho_1,\rho_2) =
\mathbf{FinRel}((X,\rho_1),(X,\rho_2))$ for any two elements 
$\rho_1,\rho_2 \in S_2(X,k)$
\end{proof}

These groupoids support functors taking $(X,\sigma)$ or $(X,\rho)$ to
set $\mathbf{Fin}(X,X)$ of permutations of $X$. The associated
Grothendieck constructions are the groupoids
whose objects are all triples
$(X,\sigma,\tau)$ or $(X,\rho,\tau)$ where $\tau$ is a permutation of
$X$. Proposition~\ref{prop:FinPR} and \cite[Proposition~2.8]{leinster08}
imply that
\begin{alignat*}{3}
  \label{eq:2}
  &\chi(\mathbf{FinPerm}_{n,k}) = S_1(n,k)/n! &&\qquad &&
  \chi(\mathbf{FinRel}_{k,m}) = S_2(k,m)/k! \\
  & \chi( \int_{\mathbf{FinPerm}_{n,k}} \mathbf{Fin}(X,X) ) = S_1(n,k) &&\qquad
  && \chi( \int_{\mathbf{FinRel}_{k,m}} \mathbf{Fin}(X,X) ) = S_2(k,m)
\end{alignat*}
where $S_1(n,k)$ and  $S_2(k,m)$
a  Stirling numbers of the first and second kind \cite[pp
18, 33]{stanley97}.

In the $O(1)$-span
 \begin{center}
      \begin{tikzpicture}[double distance = 1.5pt]
       \matrix (dia) [matrix of math nodes, column sep=65pt, row
       sep=35pt]{
       \displaystyle{\int_{\mathbf{FinPerm}} \mathbf{Fin}(X,X) }&   \Z \\
       \Z &  BO(1)\\};
       \draw[->] (dia-1-1) -- node[above]{$R$} (dia-1-2);
       \draw[->] (dia-2-1) -- (dia-2-2);
       \draw[->] (dia-1-1) -- node[left]{$L$} (dia-2-1);
       \draw[->] (dia-1-2) --  (dia-2-2);
       \path (dia-2-1) -- node[midway]{$\ve{1}(X,\sigma)=(-1)^{|X|-|\sigma|}$}(dia-1-2);
   \end{tikzpicture}
\end{center}
$L(X,\sigma,\tau)=|X|$, the cardinality of $X$, $R(X,\sigma,\tau) = |\sigma|$ the number of cycles in $\sigma$,
and the functors to $BO(1)$ are constant.  For any two natural numbers, $n$
and $k$, the two-sided homotopy fibre
\begin{equation*}
  \dqt n{\int_{\mathbf{FinPerm}} \mathbf{Fin}(X,X)}k = \int_{\mathbf{FinPerm_{n,k}}} \mathbf{Fin}(X,X)
\end{equation*}
and, since, moreover, for any $g \in O(1)$,
\begin{equation*}
   (\dqt{n}{\int_{\mathbf{FinPerm}} \mathbf{Fin}(X,X)}{k}) \{ \dqt{n}{\ve{1}}{k} = g\} =
   \begin{cases}
     \dqt{n}{\int_{\mathbf{FinPerm}} \mathbf{Fin}(X,X)}{k}  & g = (-1)^{n-k} \\
     \emptyset & g \neq (-1)^{n-k} 
   \end{cases}
\end{equation*}
the matrix entries are
\begin{equation*}
  [\int_{\mathbf{FinPerm}} \mathbf{Fin}(X,X) ,\ve{1}](n,k) = S_1(n,k)
  \cdot (-1)^{n-k} \qquad
  [[\int_{\mathbf{FinPerm}} \mathbf{Fin}(X,X) ,\ve{1}]](n,k) = (-1)^{n-k} S_1(n,k) 
\end{equation*}
in $\Q O(1)$ and $\Q$, respectively. 

Similarly, in the $O(1)$-span
\begin{center}
      \begin{tikzpicture}[double distance = 1.5pt]
       \matrix (dia) [matrix of math nodes, column sep=65pt, row
       sep=35pt]{
       \displaystyle{\int_{\mathbf{FinRel}} \mathbf{Fin}(X,X)} &   \Z  \\
       \Z &  BO(1)\\};
       \draw[->] (dia-1-1) -- node[above]{$R$} (dia-1-2);
       \draw[->] (dia-2-1) -- (dia-2-2);
       \draw[->] (dia-1-1) -- node[left]{$L$} (dia-2-1);
       \draw[->] (dia-1-2) --  (dia-2-2);
       \path (dia-2-1) -- node[midway]{$\ve{2}(X,\rho)=+1$}(dia-1-2);
   \end{tikzpicture}
\end{center}
$L(X,\rho,\tau)=|X|$, the cardinality of $X$, and
$R(X,\rho,\tau)=|\rho|$, the number of equivalence classes for the
equivalence relation $\rho$. The associated complex matrix is
\begin{equation*}
  [[\int_{\mathbf{FinRel}} \mathbf{Fin}(X,X), \ve{2}]] = S_2
\end{equation*}
since
the two-sided homotopy fibre
\begin{equation*}
  \dqt k{\int_{\mathbf{FinRel}} \mathbf{Fin}(X,X)}m = \int_{\mathbf{FinRel}_{k,m}} \mathbf{Fin}(X,X)
\end{equation*}
has \Euc\ $S_2(k,m)$ for any two natural numbers, $k$ and $m$.

In the composite $O(1)$-span from $\Z$ to $\Z$ (Definition~\ref{defn:funddefn}.\eqref{funddefn:4a})
\begin{center}
      \begin{tikzpicture}[double distance = 1.5pt]
       \matrix (dia) [matrix of math nodes, column sep=110pt, row
       sep=35pt]{
       \displaystyle{\int_{\mathbf{FinPerm} \times_{\Z}
           \mathbf{FinRel}} \mathbf{Fin}(X_1,X_1) \times \mathbf{Fin}(X_2,X_2)} &   \Z  \\
       \Z &  BO(1)\\};
       \draw[->] (dia-1-1) -- node[above]{$R$} (dia-1-2);
       \draw[->] (dia-2-1) -- (dia-2-2);
       \draw[->] (dia-1-1) -- node[left]{$L$} (dia-2-1);
       \draw[->] (dia-1-2) --  (dia-2-2);
       \path (dia-2-1) -- node[midway]{$(\ve{1} \times_{\Z} \ve{2})(X_1,\sigma,X_2,\rho,\tau_1,\tau_2)=(-1)^{|X_1|-|\sigma|}$}(dia-1-2);
   \end{tikzpicture}
\end{center}
the apex
is the category of elements of the set-valued functor
$(X_1,X_2) \to \mathbf{Fin}(X_1,X_1) \times \mathbf{Fin}(X_2,X_2)$ on the groupoid
$\mathbf{FinPerm} \times_{\Z} \mathbf{FinRel}$ of all
  $(X_1,\sigma,X_2,\rho)$ where $\sigma$ a permutation of the finite set
  $X_1$, $\rho$ is an equivalence relation on the finite set $X_2$ of cardinality
  $|X_2|=|\sigma|$.  Then $|\rho| \leq |X_2| = |\sigma| \leq |X_1|$. The
  functors $L$ and $R$ are $L(X_1,\sigma,X_2,\rho,\tau_1,\tau_2)=|X_1|$
  and $R(X_1,\sigma,X_2,\rho,\tau_1,\tau_2)=|\rho|$. The associated matrix
\begin{multline*}
  [[ \int_{\mathbf{FinPerm} \times_{\Z} \mathbf{FinRel}} \mathbf{Fin}(X_1,X_1)
  \times \mathbf{Fin}(X_2,X_2), \ve{1} \times_{\Z} \ve{2}]] \\
  =
[[\int_{\mathbf{FinPerm}} \mathbf{Fin}(X_1,X_1), \ve{1}]]
[[\int_{\mathbf{FinRel}} \mathbf{Fin}(X_2,X_2), \ve{2}]] = \big( (-1)^{n-k} S_1(n,k) \big)_{n,k}
\big( S_2(k,m) \big)_{k,m} = \big( \delta_{n,m} \big)_{n.m}
\end{multline*}
is the identity matrix by Theorem~\ref{thm:main} and
\cite[Proposition~1.4.1]{stanley97}. 
Alternatively, one
may use that
\begin{multline*}
  \dqt n{\int_{\mathbf{FinPerm} \times_{\Z}
           \mathbf{FinRel}} \mathbf{Fin}(X_1,X_1) \times
         \mathbf{Fin}(X_2,X_2)}m \\ =
       \begin{cases}
         \int_{\mathbf{Fin}_n} \mathbf{Fin}(X_1,X_1) \times 
         \int_{\mathbf{Fin}_n} \mathbf{Fin}(X_2,X_2) & n=m \\
        \coprod_{m \leq k \leq n} \int_{\mathbf{FinPerm}_{n,k}}
        \mathbf{Fin}(X_1,X_1) \times \int_{\mathbf{FinPerm}_{k,m}}
        \mathbf{Fin}(X_2,X_2) & n \neq m
       \end{cases}
\end{multline*}
to arrive at the same conclusion.


 
\def\cprime{$'$} \def\cprime{$'$} \def\cprime{$'$} \def\cprime{$'$}
  \def\cprime{$'$}
\providecommand{\bysame}{\leavevmode\hbox to3em{\hrulefill}\thinspace}
\providecommand{\MR}{\relax\ifhmode\unskip\space\fi MR }
\providecommand{\MRhref}[2]{%
  \href{http://www.ams.org/mathscinet-getitem?mr=#1}{#2}
}
\providecommand{\href}[2]{#2}


\begin{thebibliography}{1}

\bibitem{BHW2010}
John~C. Baez, Alexander~E. Hoffnung, and Christopher~D. Walker, \emph{Higher
  dimensional algebra {VII}: {G}roupoidification}, Theory Appl. Categ.
  \textbf{24} (2010), No. 18, 489--553. \MR{2770073}

\bibitem{imma_kock_tonks2018}
Imma G\'{a}lvez-Carrillo, Joachim Kock, and Andrew Tonks, \emph{Homotopy linear
  algebra}, Proc. Roy. Soc. Edinburgh Sect. A \textbf{148} (2018), no.~2,
  293--325. \MR{3777576}

\bibitem{kock_moller2026}
Joachim Kock and Jesper~M. M{\o}ller, \emph{Signs in objective linear algebra,
  exemplified with exterior powers and determinants}, Available from arXiv.

\bibitem{leinster08}
Tom Leinster, \emph{The {E}uler characteristic of a category}, Doc. Math.
  \textbf{13} (2008), 21--49,
  \href{http://www.math.uiuc.edu/documenta/vol-13/02.pdf}{Doc. Math.}
  \MR{MR2393085}

\bibitem{stanley97}
Richard~P. Stanley, \emph{Enumerative combinatorics. {V}ol. 1}, Cambridge
  Studies in Advanced Mathematics, vol.~49, Cambridge University Press,
  Cambridge, 1997, With a foreword by Gian-Carlo Rota, Corrected reprint of the
  1986 original. \MR{MR1442260 (98a:05001)}

\end{thebibliography}

\end{document}